\documentclass[preprint]{elsarticle}
\usepackage{hyperref}
\usepackage{amsmath}
\usepackage{amsthm}
\usepackage{amssymb}
\usepackage{color}
\usepackage{comment}
\usepackage{graphicx}
\usepackage{enumitem}
\usepackage{mathtools}
\usepackage{multirow}
\usepackage{multicol}
\usepackage[mathlines]{lineno}
\usepackage{bbm}
\usepackage[table]{xcolor}
\usepackage[longend,ruled]{algorithm2e}
\usepackage{float}
\usepackage{booktabs}
\usepackage{longtable}

\newcommand{\supp}{{\Phi}}
\newcommand\spec{{\mathrm{spec}}}

\newtheorem{theorem}{Theorem}[section]
\newtheorem{lema}[theorem]{Lemma}
\newtheorem{corollary}[theorem]{Corollary}
\newtheorem{proposition}[theorem]{Proposition}
\newtheorem{remark}[theorem]{Remark}
\newtheorem{definition}[theorem]{Definition}

\newtheorem{teorema}[theorem]{Theorem}

\newtheorem{problem}{Problem}

\begin{document}


\begin{frontmatter}

\title{Structured eigenbases and pair state transfer on threshold graphs}

\author[heber,Leoaddress1]{Leonardo de Lima}
\ead{leonardo.delima@ufpr.br}

\author[renata]{Renata Del-Vecchio}
\ead{rrdelvecchio@id.uff.br}

\author[HM]{Hermie Monterde}
\ead{Hermie.Monterde@uregina.ca}

\author[heber]{Heber Teixeira }
\ead{hebercristina@ufpr.br}

\address[heber]{Programa de Pós-Graduação em Matemática, Universidade Federal do Paraná, Paraná, Brasil}

\address[renata]{Instituto de Matemática e Estatística, Universidade Federal Fluminense, Niterói, Brasil}

\address[HM]{Department of Mathematics and Statistics, University of Regina, Regina, SK, Canada S4S 0A2}

\address[Leoaddress1]{Departamento de Administração Geral e Aplicada, Universidade Federal do Paraná, Paraná, Brasil}

\begin{abstract}
Recently, Macharete, Del-Vecchio, Teixeira and de Lima showed that a star and any threshold graph on the same number of vertices share the same eigenbasis relative to the Laplacian matrix. We use this fact to establish two main results in this paper. The first one is a characterization of threshold graphs that are \textit{simply structured}, i.e., their associated Laplacian matrices have eigenbases consisting of vectors with entries from the set $\{-1,0,1\}$. Then, we provide sufficient conditions such that a simply structured threshold graph is weakly Hadamard diagonalizable (WHD). This allows us to list all connected simply structured threshold graphs on at most 20 vertices, and identify those that are WHD. Second, we characterize Laplacian pair state transfer on threshold graphs. In particular, we show that the existence of Laplacian vertex state transfer and Laplacian pair state transfer on a threshold graph are equivalent if and only if it is not a join of a complete graph and an empty graph of certain sizes. 
\end{abstract}


\end{frontmatter}


\section{Introduction}\label{intro}

Unraveling graph properties through the eigenvectors of matrices associated with graphs constitutes a central topic in spectral graph theory. In particular, graphs with eigenspaces spanned by eigenvectors whose entries are restricted to the sets $\{-1,1\}$ and $\{- 1,0,1 \}$ have attracted attention in recent years \cite{WHD, AL21, ALN23, Caputo, JP25,  MDTL24, WeakMatrices}. The existence of such structurally simple eigenspaces not only presents significant computational advantages \cite{SS07}, but also offers a more straightforward eigenvector analysis \cite{SS08}. Hadamard diagonalizable (HD) graphs are an example of such a family. These are graphs whose Laplacian matrices are diagonalized by a Hadamard matrix \cite{barik2011}. Later, the concept of weakly Hadamard diagonalizable (WHD) graphs was introduced \cite{WHD}. These are graphs whose Laplacian matrices are diagonalized by a weak Hadamard matrix, which is a matrix $W$ with entries from the set $\{-1,0,1\}$ such that $W^{T}W$ is tridiagonal. Note that weak Hadamard matrices and WHD graphs generalize Hadamard matrices and HD graphs, respectively. Very recently, WHD graphs have gained interest from theoretical \cite{JP25,MDTL24} and applied perspectives \cite{WeakMatrices}. 

In \cite{MDTL24}, Macharete, Del-Vecchio, Teixeira and de Lima showed that a star and any threshold graph of the same order share the same eigenbasis. In this paper, we use this fact to establish two main results. First, we fully characterize threshold graphs that are simply structured, i.e., graphs whose Laplacian matrices have eigenbases consisting of vectors with entries in $\{-1,0,1\}$. This allows us to determine sufficient conditions for simply structured threshold graphs to be WHD. Since threshold graphs form a subclass of cographs, this provides a partial answer to an open question first raised in \cite{WHD}, which asks to characterize all cographs that are WHD. For our second result, we characterize Laplacian pair state transfer on threshold graphs, a type of quantum transport that represents accurate transmission of pairs of entangled qubit states. This contributes the literature on quantum walks on threshold graphs \cite{kirkland2011spin,kirkland2020fractional,monterde2026new}.

This paper is organized as follows. In Section \ref{sec:2}, we introduce basic notation, definitions, and preliminary results. In Section \ref{sec:3}, we characterize all threshold graphs that are simply structured. This allows us to list all connected threshold graphs on at most 20 vertices with simply structured eigenbases, see Table \ref{tab:tab1} in the Appendix. In Section \ref{sec:4} we provide sufficient conditions such that a simply structured threshold graph is WHD. In Section \ref{sec:5}, we give a complete characterization of Laplacian pair state transfer on threshold graphs. As a corollary, we show that vertex state transfer and pair state transfer on a threshold graph relative to the Laplacian are equivalent if and only if the threshold graph is not a join of a complete and empty graph of certain sizes. Finally, we discuss open questions in Section \ref{sec:6}.

\section{Preliminaries}\label{sec:2}

Throughout the paper, we let $G=(V,E)$ be a simple graph on $n$ vertices with vertex set $V=\{1, \ldots, n\}$ and edge set $E$. The degree of a vertex $i$ of $G$, denoted $d_{i}(G)$ or $d_{i}$ when the graph is clear from the context, is the number of edges incident to vertex $i$. The \textit{Laplacian matrix} of $G$ is given by
\begin{center}
$L(G) = D(G)-A(G)$,
\end{center}
where $A(G)$ is the (0,1)-adjacency matrix and $D(G)$ is the diagonal matrix of the vertex degrees of $G$. We let $\spec(G)$ denote the set of distinct Laplacian eigenvalues of $G$. The eigenvalues of $L(G)$ (including the repeated ones) are arranged in a nonincreasing order given by 
\begin{center}
$\mu_1 \geq  \cdots \geq \mu_{n-1} \geq \mu_{n} = 0.$
\end{center}
Let $G$ and $H$ be two graphs. The \textit{disjoint union} $G\sqcup H$ of $G$ and $H$ is the graph with vertex set $V(G)\sqcup V(H)$ and edge set $E(G)\sqcup E(H)$. The \textit{join} $G\vee H$ of $G$ and $H$ is the graph obtained from the disjoint union of $G$ and $H$ and adding all edges with one endpoint in $G$ and the other in $H$. We denote the complete graph, complete graph minus an edge and star on $n$ vertices by $K_n$, $K_n-e$ and $S_n$, respectively. We write the complement of $G$ by $G^c$. We also let $\mathbbm{1}_n$ denote the all-ones vector of order $n$. 

An \textit{eigenbasis} of $L(G)$ is a basis of $\mathbb{R}^n$ composed of eigenvectors of $L(G)$. A basis is \textit{simply structured} if it  consists of vectors with entries in $\{-1,0,1\}$. For an eigenvalue $\mu$ of $L(G)$, we denoted the eigenspace associated with $\mu$ by
\begin{center}
$\mathcal{E}_{L}(\mu) = \{\mathbf{x} \in \mathbb{R}^{n}, \; \mathbf{x}\ne 0: L(G)\mathbf{x} = \mu \mathbf{x} \}.$ 
\end{center}
We say that $\mathcal{E}_{L}(\mu)$ is \textit{simply structured} if it admits a simply structured basis, and a graph $G$ is \textit{simply structured} if $L(G)$ has a simply structured eigenbasis. Note that if $G$ is simply structured, then $G$ is necessarily\textit{Laplacian integral}, which is to say that $\spec(G)$ consists of all integers. In recent years, simply structured eigenspaces for certain eigenvalues of matrices are associated with graphs have been studied in \cite{AAGK06,AL21, ALN23,  SS08,SS07, SS09}. 

A matrix $W \in \mathbb{R}^{n \times n}$ is called \textit{weak Hadamard} if $W$ has all entries from the set $\{-1,0,1\}$ and $W^{T}W$ is a tridiagonal matrix \cite{WHD}. A graph $G$ is \textit{weakly Hadamard  diagonalizable} (or WHD for short) if the Laplacian matrix of $G$, denoted by $L(G)$, is diagonalizable by a weak Hadamard matrix $W$, that is, 
\begin{align*}
        L(G) = W\Lambda W^{-1},
    \end{align*}
where $\Lambda$ is the diagonal matrix of eigenvalues of $L(G).$ An analogous definition applies to \textit{Hadamard diagonalizable} graphs. Note that WHD and HD graphs are simply structured by definition. However, simply structured graphs need not be WHD nor HD (see examples in Appendix A). We also note that the definition of WHD may also be adapted to the adjacency matrix or signless Laplacian matrix of a graph. However, since these matrices are nonnegative, the Perron-Frobenius theorem forces WHD graphs relative to these matrices to have $\mathbbm{1}_n$ as an eigenvector. That is, such graphs are regular, and thus must also be WHD relative to the Laplacian matrix. Hence, the theory of WHD graphs under the Laplacian framework is less restrictive, and therefore more interesting.

In \cite{WHD}, the authors raised the problem of finding cographs that are WHD. Since any threshold graph is a cograph, we consider a related problem:

\vspace{0.2cm}

\begin{problem}\label{problem A}
    Determine all connected threshold graphs that are WHD.
\end{problem}

There are several ways to define a threshold graph, as discussed in \cite{MP95}. We utilize that which involves binary generating sequences.
\begin{definition}
\label{def:threshold}
    Let $(b_{i})$ be a sequence of $0's$ and $1's$ of length $n$. The threshold graph associated with the binary sequence $\mathbf{b}=(b_{i})$ is the graph on $n$ vertices constructed recursively as follows.
    \begin{itemize}
        \item[(i)] If $b_{i} = 0$, then add the isolated vertex $i$.
        \item[(ii)] Otherwise, add the vertex $i$ adjacent to all vertices with label less than $i$.
    \end{itemize}
\end{definition}
A vertex $i$ is an isolated vertex if $b_i=0$ and a dominating vertex if $b_i = 1.$ Clearly, $G$ is connected if and only if $b_n = 1$. We also follow the convention that $b_1=0$ for vertex 1. Consequently, we may write $\mathbf{b}=(b_{1}, b_{2}, \ldots, b_{n})$ as
\begin{center}
$\mathbf{0}^{s_1}\mathbf{1}^{t_1}\cdots \mathbf{0}^{s_r}\mathbf{1}^{t_r}$ 
\end{center}
for some positive integers $s_1,\ldots,s_r,t_1,\ldots,t_r$ satisfying $\sum_{j=1}^r(s_j+t_j)=n$. If $s_1=1$, then the connected threshold graph determined by $\mathbf{b}$ may be written as
\begin{equation}
\label{th1}
(((((K_{t_1'}\sqcup K^c_{s_2})\vee K_{t_2})\sqcup K^c_{s_3})\cdots)\sqcup K^c_{s_r})\vee K_{t_{r}}
\end{equation}
where $t_1'=t_1+1\geq 2$. However, if $s_1\geq 2$, then we may write it as
\begin{equation}
\label{th2}
(((((K_{s_1}^c\vee K_{t_1})\sqcup K_{s_2}^c)\vee K_{t_2})\cdots)\sqcup K_{s_r}^c)\vee K_{t_{r}}
\end{equation}
The expressions in (\ref{th1}) and (\ref{th2}) coincide with the forms of connected threshold graphs provided by  Kirkland and Severini in \cite{kirkland2011spin}.

The computation of the eigenvalues of the Laplacian matrix of threshold graphs in terms of vertex degrees is well-known \cite{Merris94}. More specifically, let $G$ be a threshold graph on $n$ vertices defined by $(b_{1}, b_{2}, \ldots, b_{n})$, with degree sequence $d = (d_{1} ,d_{2},\ldots, d_n)$, where $d_1\geq d_2\geq \ldots \geq d_n$, and $\text{tr}(G)=\displaystyle\sum_{i=1}^{n}b_i$. Then, the  eigenvalues of $L(G)$, denoted as $\mu_{i}$ for each $i \in \{1, 2, \ldots, n\}$, are given by:
    \begin{align}\label{spectrum  threshold}
    \mu_i = \left\{\begin{array}{rll}
        d_i +1,  & \text{if} \ \ 1 \leq i\leq \text{tr}(G),\\
        d_{i+1}, & \text{if} \ \ \text{tr}(G) + 1 \leq i \leq n-1,\\
        0,       & \text{if} \ \ i =n.
    \end{array}\right.
    \end{align}
We also remark that threshold graphs are determined by their Laplacian spectrum; that is, no two {threshold graphs} on $n$ vertices have the same spectrum {unless they are isomorphic}.

For the star $S_n$ on $n$ vertices, the vectors 
\begin{align}
\label{base da estrela}
        \mathbf{x}^{l}_{i} = \left\{\begin{array}{rll}
        1,  & \text{ if } \ i< l+1,\\
        -l, & \text{ if } \ i = l+1,\\
        0,  & \text{ if } \ l+1 <i \leq n, 
    \end{array}\right.
    \end{align}
for each $ 1\leq l < n,$ and $\mathbf{x}^{n} = \mathbbm{1}_n$ are eigenvectors of $L(S_{n})$ 
associated with eigenvalues  $1$, $n$ and $0$ of multiplicities $n-2$, $1$ and $1$, respectively. In \cite{MDTL24}, the following result is established.

\begin{theorem}
 \label{TeoPrincipal} 
Let $G$ be a connected graph on $n$ vertices. Then $G$ is a threshold graph if and only if  $\{\mathbf{x}^{1}, \ldots, \mathbf{x}^{n}\}$ is an eigenbasis for $L(G)$. 
\end{theorem}


\section{Characterization of simply structured threshold graphs}
\label{sec:3}

To solve Problem \ref{problem A}, we first characterize threshold graphs whose every eigenspace is simply structured. In this section, we establish a lower bound on the dimension of a  simply structured eigenspace for $L(G)$, and then derive a necessary and sufficient condition that provides a complete description of the threshold graphs that are simply structured.

\subsection{Dimension of eigenspaces}

 In this subsection, we address the following problem.

\vspace{0.3cm}

\begin{problem}\label{problema A1}
    Determine the minimum number of vectors in a basis of an eigenspace of the Laplacian matrix of a connected threshold graph $G$ such that the basis is simply structured.
\end{problem}

\vspace{0.3cm}

From Theorem \ref{TeoPrincipal}, $L(G)$ and $L(S_n)$ share an eigenbasis $B=\{\mathbf{x}^{1},\ldots,\mathbf{x}^{n}\}$ for any threshold $G$ of order $n$. It is clear that the Laplacian eigenbasis of a graph is not unique, and many  other Laplacian eigenbasis $B^{\prime}$ can be obtained from $B$ through linear combinations.

Amongst all vectors in $B$, only $\mathbf{x}^{1}$ and $\mathbf{x}^{n}$ have all entries in $\{-1,0,1\}$. In particular, $0$ is a simple eigenvalue of $L(G)$. Hence, $\mathcal{E}_{L}(0)=\mathrm{span} \ \{ \mathbf{x}^{n}\}$, where $\mathbf{x}^{n} = \mathbbm{1}_n$, and so $\mathcal{E}_{L}(0)$ is simply structured. Now, since the largest eigenvalue of $L(G)$ is $n$, we address Problem \ref{problema A1} for this eigenvalue. 

\begin{proposition}\label{prop:n}
    Let $G$ be a connected threshold graph on $n$ vertices with binary sequence $\mathbf{b}=(b_1, \ldots, b_{l_1}, b_{l_1+1}, \ldots, b_{n-1}, b_n)$, and let $l_1 \in \{1,\ldots,n-1\}$ be such that $b_{l_1+1} = \cdots = b_{n}$ and  $b_{l_1}\neq b_{l_1+1}$. Then, $\mathcal{E}_L(n)$ is simply structured if and only if $1\leq l_1 \leq \lfloor \frac{n}{2} \rfloor.$
\end{proposition}

\begin{proof}
Let $l_1 \in \{1,\ldots,n-1\}$ such that $b_{l_1+1}=\cdots=b_n$ and $b_{l_1}\neq b_{l_1+1}$. Since $G$ is a connected threshold graph, $b_n=1$, and so $b_{l_1+1}=\cdots=b_n=1$ and $b_{l_1}=0$. Under these conditions, we have $\text{tr}(G)\geq n-l_1$ and from \eqref{spectrum  threshold}, it follows that $n$ is an eigenvalue of $L(G)$ with multiplicity $n-l_1$. Thus,
\begin{align*}
    \mathcal{E}_L(n)&=\mathrm{span} \ \{ \mathbf{x}^{l_1}, \ldots, \mathbf{x}^{n-1} \} \quad  \textrm{and} \quad \mathrm{dim}(\mathcal{E}_{L}(n)) = n-l_1.
\end{align*} 
Suppose that $\mathcal{E}_L(n)$ is simply structured. Then, by hypothesis, there is a basis $B=\{\mathbf{v}^{l_1+1}, \ldots, \mathbf{v}^{n-1}, \mathbf{v}^{\prime}\}$ such that 
$$\mathcal{E}_L(n)=\mathrm{span} \ \{ \mathbf{x}^{l_1}, \ldots, \mathbf{x}^{n-1} \} = \mathrm{span}\{\mathbf{v}^{l_1+1}, \ldots, \mathbf{v}^{n-1},\mathbf{v}^{\prime}\},$$
where each vector in $B$ has entries only in $\{-1,0,1\}$.
Now, let $\mathbf{v}=(v_{1}, \ldots, v_{n}) \in B$. Note that writing $\mathbf{v}$ as a linear combination of the $\mathbf{x}^{i}$'s is equivalent to finding a vector $a=(a_{l_1}, a_{l_1+1}, \ldots, a_{n-1})^{T}$ of coefficients such that 
\begin{equation}\label{eq:sistemalinear(n)}
    X a = \mathbf{v},
\end{equation}
where 
\begin{align*}
    X=[\mathbf{x}^{l_1}\cdots\mathbf{x}^{n-1}]=\left[\begin{array}{ccccc}
        1          & 1          & \ldots     & 1          & 1 \\ 
        \vdots     & \vdots     & \ddots     & \vdots     & \vdots \\ 
        1          & 1          & \ldots     & 1          & 1 \\
        -l_1         & 1          & \ldots     & 1          & 1 \\ 
        0          & -(l_1+1)     & \ldots     & 1          & 1 \\ 
        0          & 0          & \ldots     & 1          & 1 \\
        \vdots     & \vdots     & \ddots     & \vdots     & \vdots \\ 
        0          & 0          & \ldots     & 1          & 1 \\ 
        0          & 0          & \ldots     & -(n-2)     & 1 \\ 
        0          & 0          & \ldots     & 0          & -(n-1) 
    \end{array}\right].
\end{align*}
 
Note that $\mathbf{v}$ must have the same number of $-1$'s and $1$'s since it is orthogonal to $ \mathbbm{1}_n$. Since the first $l_1$ equations of \eqref{eq:sistemalinear(n)}  are all identical, we get $v_{1}=v_{2}=\cdots=v_{l_1}.$ The other $n-l_1$ equations are given by
\begin{equation}  \label{somas(n)}
\left\{\begin{array}{rrrrrrrrrrrr}
  -l_1 \, a_{l_1} + \hspace{1cm}  a_{l_1+1} + \hspace{0.2cm} \ldots \hspace{0.2cm} + \hspace{1cm} a_{n-2} + \hspace{1cm} a_{n-1} = & v_{l_1+1}   \\
                  -(l_1+1) \, a_{l_1+1} + \hspace{0.2cm} \ldots \hspace{0.2cm} + \hspace{1cm} a_{n-2} + \hspace{1cm} a_{n-1} = & v_{l_1+2}   \\
                                    &             \vdots                            \\
                                         -(n-2) \, a_{n-2}  + \hspace{1cm} a_{n-1}   = & v_{n-1}   \\
                                                       -(n-1) \, a_{n-1}  = & v_{n}
\end{array}\right.
\end{equation}

Let us analyze the cases where $v_1 = \cdots =v_{l_1} = 0$ or $v_1 = \cdots = v_{l_1} = 1.$ Note that the case $v_1 = \cdots = v_{l_1} = -1$  is analogous to the latter case, except for a change of sign.

Let us first deal with the case $v_1 = \cdots =v_{l_1} = 0.$ Since the number of $1$'s and $-1$'s are equal for each $i \in \{l_1+1,\ldots,n-1\}$, take $v_{i} = 1, v_{i+1} = -1$ and $v_{j} = 0$ for all $j \ne \{i, i+1\}.$ The solution to  \eqref{eq:sistemalinear(n)} is unique and is given by
\begin{align*}
    a_{i}=\frac{1}{i}, \ a_{i+1} = -\frac{1}{i} \ \text{and} \ a_{j}=0, \ \text{for all} \  j \in \{l_1+1, \ldots, n-1\} \setminus \{i, i+1\}.  
\end{align*}
This shows that it is possible to combine the vectors $\mathbf{x}^{l_1+1},\ldots,\mathbf{x}^{n-1}$ to obtain $n-l_1-1$ eigenvectors of $L(G)$ that are linearly independent with entries in $\{-1,0,1\}$. These vectors are
\begin{align}\label{vi estruturado(n)}
    \mathbf{v}^{i} = \frac{1}{i}\mathbf{x}^i-\frac{1}{i}\mathbf{x}^{i-1} = e_i-e_{i+1}
\end{align}
for each $i \in \{l_1+1,\ldots,n-1\}.$
Furthermore, in the other cases where there are two or more $1$'s and $-1$'s, the vectors obtained are a linear combination of the vectors $\mathbf{v}^{i}$ obtained in \eqref{vi estruturado(n)}.

We now deal with the case $v_{1} = \cdots = v_{l_1} =1.$ In this case, we must have at least $l_1$ entries equal to $-1$ among $v_{l_1+1},\ldots,v_{n}$ in the $n-l_1$  equations in \eqref{somas(n)}. Hence, we must have  $n-l_1 \geq l_1,$ which implies that $1\leq l_1 \leq \lfloor \frac{n}{2} \rfloor.$ Suppose that $v_{l_1+1} = \cdots = v_{2l_1} =-1$ and let  $\mathbf{v}^{\prime}=\sum_{i=1}^{l_1} e_i-\sum_{i=l_1+1}^{2l_1} e_{i}$. The solution to the system \eqref{eq:sistemalinear(n)} is unique and is given by
\begin{align*}
    a_i=\frac{2l_1}{i(i+1)}, \ \text{for} \ \  i=l_1, \ldots, 2l_1-1 \ \text{and} \ a_i=0, \ \text{for} \ i=2l_1, \ldots, n-1.
\end{align*}
Thus, it is possible to combine the vectors  $\mathbf{x}^{l_1},\ldots,\mathbf{x}^{2l_1-1}$ to obtain an eigenvector of $L(G)$ with entries in $\{-1,0,1\}$ given by
\begin{align}\label{v' estruturado(n)}
    \mathbf{v}^{\prime}=\frac{2l_1}{l_1(l_1+1)} \, \mathbf{x}^{l_1} + \ldots + \frac{2l_1}{(2l_1-2)(2l_1-1)} \, \mathbf{x}^{2l_1-2} + \frac{1}{2l_1-1} \, \mathbf{x}^{2l_1-1}
\end{align}
Note that $\mathbf{v}^{\prime}$ in \eqref{v' estruturado(n)} is not a linear combination of the vectors $\mathbf{v}^{l_1+1}, \ldots, \mathbf{v}^{n-1}$ in \eqref{vi estruturado(n)}, since the first $l_1$ coordinates of $\mathbf{v}^{\prime}$ are equal to 1, while the first $l_1$ coordinates of  $\mathbf{v}^{l_1+1}, \ldots, \mathbf{v}^{n-1}$ are equal to 0. Thus, $\{\mathbf{v}^{l_1+1},\ldots,\mathbf{v}^{n-1}, \mathbf{v}^{\prime}\}$ is a linearly independent set. Hence, if $\mathcal{E}_L(n)$ is simply structured, then $1\leq l_1 \leq \lfloor \frac{n}{2} \rfloor.$

Conversely, suppose $1\leq l_1 \leq \lfloor \frac{n}{2} \rfloor$ with $b_{l_1+1}=\cdots=b_n=1$ and $b_{l_1}=0$. Then, 
\begin{center}
$\mathcal{E}_L(n)=\mathrm{span} \ \{ \mathbf{x}^{l_1}, \ldots, \mathbf{x}^{n-1} \} \quad  \textrm{and} \quad \left\lceil \frac{n}{2} \right\rceil \leq \mathrm{dim}(\mathcal{E}_{L}(n)) \leq n-1.$
\end{center} 
Since $l_1 \in \{1,\ldots, \lfloor \frac{n}{2} \rfloor\}$, it is possible to combine the vectors  $\mathbf{x}^{l_1}, \ldots, \mathbf{x}^{n-1}$ that are in the basis of $\mathcal{E}_L(n)$  to obtain a new basis for $\mathcal{E}_L(n)$ formed by the vectors $\mathbf{v}^{l_1+1}, \ldots, \mathbf{v}^{n-1},\mathbf{v}^{\prime}$ with entries only in $\{-1,0,1\}$, where the vectors $\mathbf{v}^{i}$, para $i=l_1+1, \ldots, n-1$ are given in \eqref{vi estruturado(n)} and the vector  $\mathbf{v}^{\prime}$ is given in \eqref{v' estruturado(n)}. Therefore, if $1\leq l_1 \leq \lfloor \frac{n}{2} \rfloor$, then  $\mathcal{E}_L(n)$ is simply structured.
\end{proof} 

From Proposition \ref{prop:n}, we conclude that the minimum number of vectors in the basis of $\mathcal{E}_{L}(n)$ is $\left\lceil\frac{n}{2}\right\rceil$. This result provides the answer to Problem \ref{problema A1} for $\mathcal{E}_{L}(n).$ Assuming that $G$ is a connected threshold graph with binary sequence $\mathbf{b}=(b_1, \ldots, b_{l_1}, b_{l_1+1}, \ldots, b_{n-1}, b_n)$ such that $b_{l_1+1}=\cdots=b_{n}=1$ and $b_{l_1}=0$, for some $l_1\in \{2,\ldots, n-1\}$, we always obtain $n-l_1$ as an eigenvalue of $L(G).$ For this reason, we will now focus on addressing Problem \ref{problema A1} for the eigenspace associated with the eigenvalue $n-l_1$. 

\begin{proposition}\label{prop:n-k}
   Let $G$ be a connected threshold graph on $n$ vertices with binary sequence $\mathbf{b}=(b_1, \ldots, b_{l_2}, b_{l_2+1}, \ldots, b_{l_1-1}, b_{l_1}, b_{l_1+1}, \ldots, b_{n-1}, b_{n})$ such that $b_{l_1}=0$ and $b_{l_1+1}=\cdots=b_{n}=1$, for some $l_1 \in \{1,\ldots, n-1\}$. Let $l_2\in \{2, \ldots, l_1-1\}$ such that 
   $b_{l_2+1}=\cdots=b_{l_1}$ and $b_{l_2}\neq b_{l_2+1}$. Then $\mathcal{E}_{L}(n-l_1)$ is simply structured if and only if $2\leq l_2 \leq \lfloor \frac{l_1}{2} \rfloor$.
\end{proposition}

\begin{proof}
By Theorem \ref{TeoPrincipal}, $L(G)$ and $L(S_n)$ share the same eigenvectors. By hypothesis,  $b_{l_1}=0$ and $b_{l_1+1}=\cdots=b_{n}=1$, for some $l_1 \in \{1, \ldots, n-1\}$. Let $l_2 \in \{2, \ldots,l_1-1\}$ be such that $b_{l_2+1}=\cdots=b_{l_1}=0$ and $b_{l_2}=1$. Under these conditions, we have $\text{tr}(G)\geq n-l_1$ and from \eqref{spectrum  threshold}, it follows that $n-l_1$ is an eigenvalue of $L(G)$ with multiplicity $l_1-l_2$. By Theorem \ref{TeoPrincipal}, we have that  
\begin{align*}
    \mathcal{E}_L(n-l_1)&=\mathrm{span} \ \{ \mathbf{x}^{l_2}, \ldots, \mathbf{x}^{l_1-1} \} \quad  \textrm{and} \quad \mathrm{dim}(\mathcal{E}_{L}(n-l_1)) = l_1-l_2.
\end{align*} 
Suppose that $\mathcal{E}_L(n-l_1)$ is simply structured. Then, by hypothesis, there is a basis $B=\{\mathbf{u}^{l_2+1}, \ldots, \mathbf{u}^{l_1-1}, \mathbf{u}^{\prime}\}$ such that 
$$\mathcal{E}_L(n-l_1)=\mathrm{span} \ \{ \mathbf{x}^{l_2}, \ldots, \mathbf{x}^{l_1-1} \} = \mathrm{span}\{\mathbf{u}^{l_2+1}, \ldots, \mathbf{u}^{l_1-1},\mathbf{u}^{\prime}\},$$
where each vector in $B$ has entries only in the set $\{-1,0,1\}$.
Take $\mathbf{u}=(u_{1}, \ldots, u_{n}) \in B$ and determining it from the vectors $\mathbf{x}^{i}$ corresponds to finding coefficients $a=(a_{l_2}, a_{l_2+1}, \ldots, a_{l_1-1})^{T}$ such that 
\begin{equation}\label{eq:sistemalinear(n-k)}
    X a = \mathbf{u},
\end{equation}
where 
\begin{align*}
    X=[\mathbf{x}^{l_2}\cdots\mathbf{x}^{l_1-1}]=\left[\begin{array}{ccccc}
        1          & 1            & \ldots      & 1           & 1 \\ 
        \vdots     & \vdots       & \ddots      & \vdots      & \vdots \\ 
        1          & 1            & \ldots      & 1           & 1 \\
        -l_2         & 1            & \ldots      & 1           & 1 \\ 
        0          & -(l_2+1)       & \ldots      & 1           & 1 \\ 
        0          & 0            & \ldots      & 1           & 1 \\
        \vdots     & \vdots       & \ddots      & \vdots      & \vdots \\ 
        0          & 0            & \ldots      & 1           & 1 \\ 
        0          & 0            & \ldots      & -(l_1-2)    & 1 \\ 
        0          & 0            & \ldots      & 0           & -(l_1-1) \\
        0          & 0            & \ldots      & 0           & 0 \\ 
        \vdots     & \vdots       & \ddots      & \vdots      & \vdots \\ 
        0          & 0            & \ldots      & 0           & 0  
    \end{array}\right]. 
\end{align*}

Note that  $\mathbf{u}$ must have same number of $-1'$s and $1'$s since it is orthogonal to $\mathbbm{1}_n$. 
Observe that the first $l_2$ equations of \eqref{eq:sistemalinear(n-k)} are all identical, and therefore $u_{1}=u_{2}=\cdots=u_{l_2}.$ The other $l_1-l_2$ equations are given by 
\begin{equation}  \label{somas(n-k)}
\left\{\begin{array}{rrrrrrrrrrrr}
  -l_2 \, a_{l_2} + \hspace{1cm}  a_{l_2+1} + \hspace{0.2cm} \ldots \hspace{0.2cm} + \hspace{1cm} a_{l_1-2} + \hspace{1cm} a_{l_1-1} = & u_{l_2+1}   \\
                  -(l_2+1) \, a_{l_2+1} + \hspace{0.2cm} \ldots \hspace{0.2cm} + \hspace{1cm} a_{l_1-2} + \hspace{1cm} a_{l_1-1} = & u_{l_2+2}   \\
                                                                                       & \vdots   \\
                                         -(l_1-2) \, a_{l_1-2}  + \hspace{1cm} a_{l_1-1}   = & u_{l_1-1}   \\
                                                       -(l_1-1) \, a_{l_1-1}  = & u_{l_1}
\end{array}\right.
\end{equation}
and the remaining $n-l_1$ equations from  \eqref{eq:sistemalinear(n-k)} are all equal to 0, and therefore $u_{l_1+1}=\cdots=u_{n}=0$. 

Let us  analyze the cases where $u_1 = \cdots =u_{l_2} = 0$ or $u_1 = \cdots = u_{l_2} = 1.$ Note that the case where $u_1 = \cdots = u_{l_2} = -1$ is analogous to the previous case, except for the change of sign. 

First, suppose $u_1 = \cdots =u_{l_2} = 0.$ Since the number of $1$'s and $-1$'s are equal, for each $i \in \{l_2+1,\ldots,l_1-1\}$, take $u_{i} = 1, u_{i+1} = -1$ and $u_{j} = 0$ for all $j \ne \{i, i+1\}.$ The solution to the system \eqref{eq:sistemalinear(n-k)} is unique and given by 
\begin{align*}
    a_{i}=\frac{1}{i}, \ a_{i+1} = -\frac{1}{i} \ \text{and} \ a_{j}=0, \ \text{for} \  j \in \{l_2+1, \ldots, l_1-1\} \setminus \{i, i+1\}.    
\end{align*}
This shows that is possible to combine the vectors  $\mathbf{x}^{l_2+1},\ldots,\mathbf{x}^{l_1-1}$ to obtain $l_1-l_2-1$ eigenvectors of $L(G)$ that are linearly independent with entries in $\{-1,0,1\}$. These vectors are 
\begin{align}\label{ui estruturado(n)}
    \mathbf{u}^{i} = \frac{1}{i}\mathbf{x}^i-\frac{1}{i}\mathbf{x}^{i-1} = e_i-e_{i+1}
\end{align}
for each $i \in \{l_2+1,\ldots,l_1-1\}.$
Furthermore, in the other cases where are two or more $1'$s and $-1'$s, the vectors obtained are a linear combination of the vectors $\mathbf{u}^{i}$ obtained in \eqref{ui estruturado(n)}.

Now, suppose that $u_{1} = \cdots = u_{l_2} =1.$ In this case, we must have at least $l_2$ entries equal to $-1$ among $u_{l_2+1},\ldots,u_{l_1}$ in the $l_1-l_2$ equations in \eqref{somas(n-k)}. Hence, we have must $l_1 -l_2 \geq l_2,$ which implies that $2\leq l_2 \leq \lfloor \frac{l_1}{2} \rfloor.$ Suppose that $u_{l_2+1} = \cdots = u_{2l_2} =-1$ and let $\mathbf{u}^{\prime}=\sum_{i=1}^{l_2} e_i-\sum_{i=l_2+1}^{2l_2} e_{i}$. The solution to the system \eqref{eq:sistemalinear(n-k)} is unique and is given by 
\begin{align*}
    a_i=\frac{2l_2}{i(i+1)}, \ \text{for} \ \  i=l_2, \ldots, 2l_2-1 \ \text{and} \ a_i=0, \ \text{for} \ i=2l_2, \ldots, l_1-1.
\end{align*}
Thus, it is possible to combine the vectors $\mathbf{x}^{l_2},\ldots,\mathbf{x}^{2l_2-1}$ to obtain $1$ eigenvector of $L(G)$ with entries in $\{-1,0,1\}$ given by
\begin{align}\label{u' estruturado(n)}
    \mathbf{u}^{\prime}=\frac{2l_2}{l_2(l_2+1)} \, \mathbf{x}^{l_2} + \cdots + \frac{2l_2}{(2l_2-2)(2l_2-1)} \, \mathbf{x}^{2l_2-2} + \frac{1}{2l_2-1} \, \mathbf{x}^{2l_2-1}
\end{align}
Note that $\mathbf{u}^{\prime}$ in \eqref{u' estruturado(n)} is not a linear combination of the vectors  $\mathbf{u}^{l_2+1}, \ldots, \mathbf{u}^{l_1-1}$ in \eqref{ui estruturado(n)}, since the first $l_2$ coordinates of $\mathbf{u}^{\prime}$ are equal to 1, while the first $l_2$ coordinates of $\mathbf{u}^{l_2+1}, \ldots, \mathbf{u}^{l_1-1}$ are equal to 0. Therefore, the vectors in the set $\{\mathbf{u}^{l_2+1},\ldots,\mathbf{u}^{l_1-1}, \mathbf{u}^{\prime}\}$ are linearly independent. Thus, if $\mathcal{E}_L(n-l_1)$ is simply structured, then $2\leq l_2 \leq \lfloor \frac{l_1}{2} \rfloor.$

On the other hand, suppose that $2\leq l_2 \leq \lfloor \frac{l_1}{2} \rfloor$ and that $b_{l_2+1}=\cdots=b_{l_1}=0$ and $b_{l_2}=1$. So,
\begin{center}
$\mathcal{E}_L(n-l_1)=\mathrm{span} \ \{ \mathbf{x}^{l_2}, \ldots, \mathbf{x}^{l_1-1} \} \quad  \textrm{and} \quad \left\lceil \frac{l_1}{2} \right\rceil \leq \mathrm{dim}(\mathcal{E}_{L}(n-l_1)) \leq l_1-2.$
\end{center} 
Since $l_2 \in \{2,\ldots, \lfloor \frac{l_1}{2} \rfloor\}$, it is possible to combine the vectors $\mathbf{x}^{l_2}, \ldots, \mathbf{x}^{l_1-1}$ that are in the basis of $\mathcal{E}_L(n-l_1)$ to obtain a new basis for $\mathcal{E}_L(n-l_1)$ formed by the vectors $\mathbf{u}^{l_2+1}, \ldots, \mathbf{u}^{l_1-1},\mathbf{u}^{\prime}$ with entries only in $\{-1,0,1\}$, where the vectors $\mathbf{u}^{i}$, for $i=l_2+1, \ldots, l_1-1$ are given in \eqref{ui estruturado(n)} and the vector $\mathbf{u}^{\prime}$ is given in \eqref{u' estruturado(n)}. Therefore, if $2\leq l_2 \leq \lfloor \frac{l_1}{2} \rfloor$, then $\mathcal{E}_L(n-l_1)$ is simply structured.
\end{proof} 

From Proposition \ref{prop:n-k}, we conclude that the minimum number of vectors in the basis of $\mathcal{E}_{L}(n-l_1)$ is $\left\lceil\frac{l_1}{2}\right\rceil$. This result provides the answer to Problem \ref{problema A1} for $\mathcal{E}_{L}(n-l_1).$ Now, let $G$ be a connected threshold graph with binary sequence \break 
$\mathbf{b}=(b_1, \ldots, b_{l_2-1}, b_{l_2}, b_{l_2+1}, \ldots, b_{l_1}, b_{l_1+1}, \ldots, b_{n-1}, b_n)$ 
such that $b_{l_2}=1$, \break $b_{l_2+1}=\cdots=b_{l_1}=0$ and $b_{l_1+1}=\cdots=b_{n}=1$, for some $l_1\in \{1,\ldots, \lfloor\frac{n}{2}\rfloor\}$ and $l_2\in \{2,\ldots, \lfloor\frac{l_1}{2}\rfloor\}$. Then $\mathcal{E}_L(n)$ and $\mathcal{E}_L(n-l_1)$ are always  simply structured, regardless of the positions $b_{1}, \ldots, b_{l_2-1}$ of the binary sequence $\mathbf{b}$. Therefore, Propositions $\ref{prop:n}$ and \ref{prop:n-k} provide an important starting point to determine the threshold graphs that are simply structured. 


\subsection{Simply structured threshold graphs}\label{section SS}

\bigskip

Using the results from the previous subsection, we derive necessary and sufficient conditions for a threshold graph to be simply structured. 

Let $G$ be a connected threshold graph on $n$ vertices with binary sequence given by $\mathbf{b}=\mathbf{0}^{s_1}\mathbf{1}^{t_1}\cdots\mathbf{0}^{s_r}\mathbf{1}^{t_r}.$ Since $n=\sum_{i=1}^r (s_i+t_i)$ we have that $1\leq r\leq \lfloor\frac{n}{2}\rfloor$. Partition $V(G)$ as $V(G)=U_1\sqcup V_1 \sqcup \ldots \sqcup U_r \sqcup V_r$ where the set $U_i$ consists of the $i$th group of consecutive isolated vertices in the construction of $G$ and thus $|U_i|=s_i$, and similarly, $V_i$ consists of the $i$th group of dominating vertices in the construction of $G$ and thus $|V_i|=t_i$. If $s_1\geq 2$ then $\{U_1, V_1, \ldots, U_r, V_r\}$ is the degree partition of $G$ while if $s_1=1$ then the degree partition is $\{U_1\sqcup V_1, U_2, V_2, \ldots, U_r, V_r\}$. In any case, each subset $U_i$ is a coclique and each subset $V_i$ is a clique \cite{antiregular}. Figure \ref{threshold} illustrates the degree partition of a threshold graph;
a dashed line between $U_i$ and $V_j$ indicates that all vertices in $U_i$ are adjacent to all vertices in $V_j$, and the rectangle indicates that $V_1 \sqcup \cdots \sqcup V_r$ is a clique.

The degree sequence of $G$ is $d=(d_{t_r}, d_{t_{r-1}}, \cdots, d_{t_1}, d_{s_1}, \cdots, d_{s_{r-1}}, d_{s_r})$ where $d_{t_r} \geq \cdots \geq d_{t_1} \geq d_{s_1} \geq \cdots \geq d_{s_r}$. Moreover, each block $V_i$ of $1$'s with size $t_i$ and each block $U_i$ of $0$'s with size $s_i$, we have 
\begin{align}\label{degrees}
    d_{t_i} = n - 1 - \displaystyle\sum_{j=i+1}^r s_j \ \ \ \text{and} \ \ \  d_{s_i} = \displaystyle\sum_{j=i}^r t_j, \ \ \ \text{for} \  1\leq i\leq r,  
\end{align}
respectively, and $\text{tr}(G)=\sum_{i=1}^{r}t_r$.

\begin{lema} 
\label{specG}
Let $G$ be a threshold graph on $n$ vertices with  $\mathbf{0}^{s_1}\mathbf{1}^{t_1}\cdots \mathbf{0}^{s_r}\mathbf{1}^{t_r}$. If $s_1=1$, then 
\begin{align*}
    \spec(G) = \Big\{ 
    & n^{(t_r)}, (n-s_r)^{(t_{r-1})}, \cdots, (n-(s_2+\cdots+s_r))^{(t_1)}, \\
    & (t_2+\cdots+t_r)^{(s_2)}, \cdots, (t_{r-1}+t_r)^{(s_{r-1})}, (t_r)^{(s_r)}, \mathbf{0}^{(1)} \Big\}.
\end{align*}
If $s_1 \geq 2,$ then 
\begin{align*}
    \spec(G) = \Big\{ 
    & n^{(t_r)}, (n-s_r)^{(t_{r-1})}, \cdots, (n-(s_2+\cdots+s_r))^{(t_1)}, \\
    & (t_1+\cdots+t_r)^{(s_1-1)}, (t_2+\cdots+t_r)^{(s_2)}, \cdots, \\ 
    & (t_{r-1}+t_r)^{(s_{r-1})}, (t_r)^{(s_r)}, \mathbf{0}^{(1)} \Big\}.
\end{align*}
Moreover, the exponents represent eigenvalue multiplicities.
\end{lema}

\begin{proof}
This follows by combining \eqref{spectrum  threshold} and \eqref{degrees}.
\end{proof}

\begin{figure}[H]
    \centering
    \includegraphics[scale=0.17]{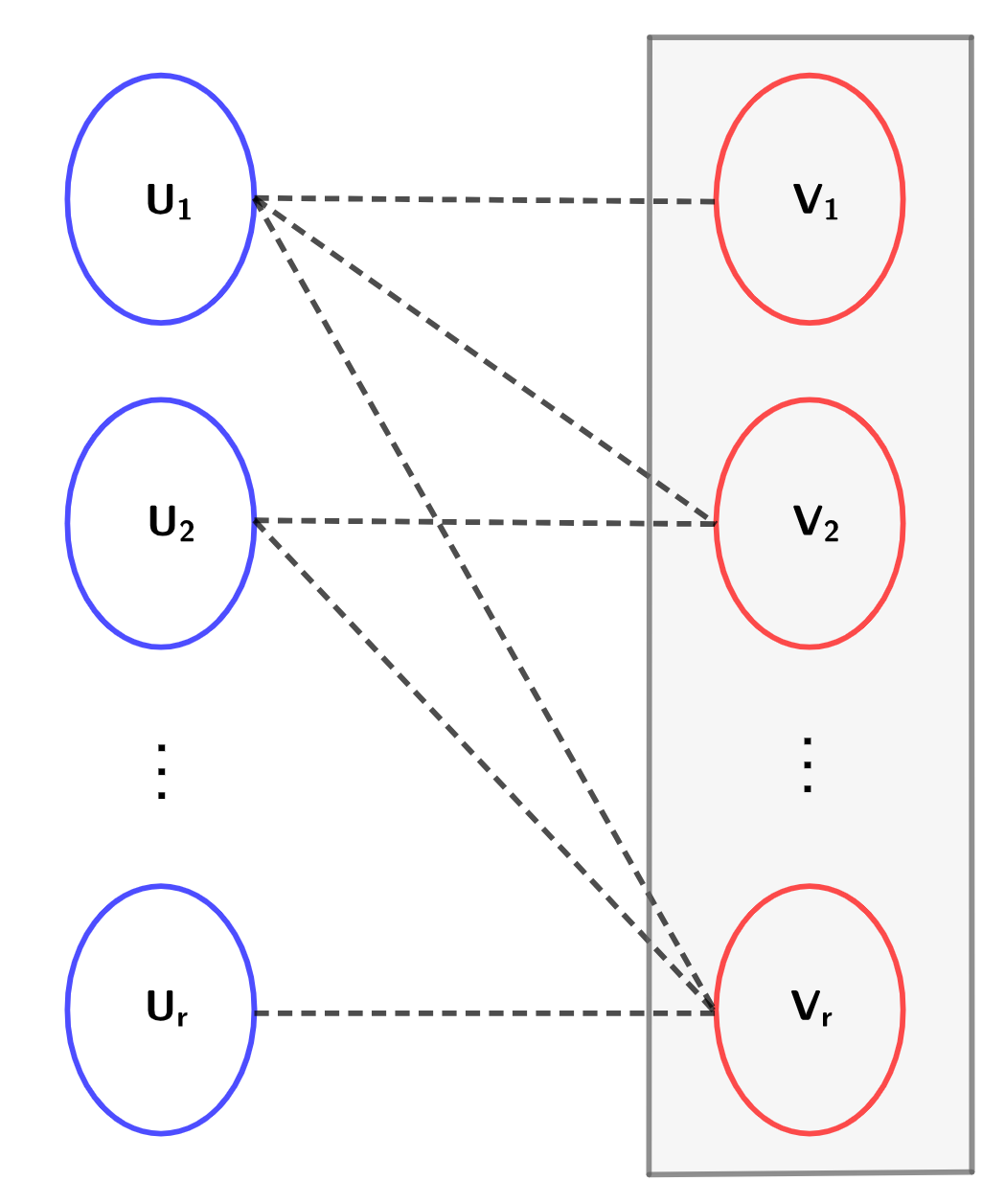}
    \caption{{\small The structure of a threshold graph with  $\mathbf{b}=\mathbf{0}^{s_1}\mathbf{1}^{t_1}\cdots \mathbf{0}^{s_r}\mathbf{1}^{t_r}$. Each vertex in $U_i$ is adjacent to $V_i \sqcup \cdots \sqcup V_r$, $V_1 \sqcup \cdots \sqcup V_r$ is a clique and $U_1 \sqcup \cdots \sqcup U_r$ is a coclique.}}
\label{threshold}
\end{figure}

The next theorem characterizes when a connected threshold graph of a given order $n$ is simply structured.

\begin{teorema}\label{teorema ss}
Let $G$ be a threshold graph on $n$ vertices with $\mathbf{b}=\mathbf{0}^{s_1}\mathbf{1}^{t_1}\cdots \mathbf{0}^{s_r}\mathbf{1}^{t_r}$. Then, $G$ is simply structured if and only if  
\begin{align}
    1\leq r \leq \left\lfloor \frac{\left\lfloor \log_2(n)  \right\rfloor + 1}{2} \right\rfloor,
\end{align}
for each $i=1, \ldots, r$,
    \begin{align}\label{s1-ss}
        \left\lceil \frac{n-\sum_{j=i+1}^r(s_j+t_j)}{2} \right\rceil \leq & \ \ t_i \ \leq n-\sum_{j=i+1}^r(s_j+t_j) -2^{2i-2}, 
    \end{align}
and for each $i=2, \ldots, r$,
\begin{align}\label{s1-ss-2}
        \left\lceil \frac{n-t_i-\sum_{j=i+1}^r(s_j+t_j)}{2} \right\rceil \leq & \ \ s_i \ \leq n-t_i-\sum_{j=i+1}^r(s_j+t_j) -2^{2i-3}, 
    \end{align}
with $s_1 = n-t_1-\sum_{j=2}^r(s_j+t_j)$.
\end{teorema}

\begin{proof}
Applying Propositions $\ref{prop:n}$ and \ref{prop:n-k} recursively, we obtain that each $l_{i} $ is bounded above by $\left\lfloor \frac{n}{2^{i}} \right\rfloor.$ 
To determine the maximum possible value to $r$ we choose the maximum value to each $l_{i}.$ Thus we can write
\begin{align*}
    n=\left\lceil \frac{n}{2} \right\rceil + \left\lceil \frac{n}{2^2} \right\rceil + \left\lceil \frac{n}{2^3} \right\rceil + \left\lceil \frac{n}{2^4} \right\rceil + \cdots + \left\lceil \frac{n}{2^{2m-1}} \right\rceil + \left\lfloor \frac{n}{2^{2m-1}} \right\rfloor,
\end{align*}
where $\left\lfloor \frac{n}{2^{2m-1}} \right\rfloor$ should be at least 1.  Thus, if $\left\lfloor \frac{n}{2^{2m-1}} \right\rfloor = 1$, we easily obtain that $m=\left\lfloor \frac{\left\lfloor \log_2(n)  \right\rfloor + 1}{2} \right\rfloor$. It implies that $1\leq r \leq \left\lfloor \frac{\left\lfloor \log_2(n)  \right\rfloor + 1}{2} \right\rfloor.$

By choosing $r$ such that $1\leq r \leq \left\lfloor \frac{\left\lfloor \log_2(n)  \right\rfloor + 1}{2} \right\rfloor$, and from  Propositions $\ref{prop:n}$ and \ref{prop:n-k}, for each $i=1,\ldots,2r-1$ we get that 
\begin{equation}\label{in:li}
    2^{2r-1-i} \leq l_{i} \leq \left\lfloor \frac{n}{2^{i}} \right\rfloor. 
\end{equation}
Since $t_r = \mathrm{dim}(\mathcal{E}_{L}(n)) = n - l_1 $, $s_r = \mathrm{dim}(\mathcal{E}_{L}(n-l_1)) = l_1-l_2$, we can relate the $s_i$ and $t_i$ to the $l_i$'s as follows:
\begin{align}\label{eq:stl}
    \begin{array}{llll}
     s_1 & = & l_{2r-1},  \\
     s_i & = & l_{2(r-i)+1} - l_{2(r-i)+2}, & \text{for} \ \ 2\leq i \leq r,\\
     t_i & = & l_{2(r-i)} - l_{2(r-i)+1}, & \text{for} \ \ 1\leq i \leq r-1,\\
     t_r & = & n - l_{1}.  \\
\end{array}
\end{align}
Replacing each $l_i$ in \eqref{in:li} by the relations of \eqref{eq:stl}, we obtain

\begin{align*}
    2^{2r-2} & \leq n-t_r \leq \left\lfloor \frac{n}{2} \right\rfloor \\ 
    2^{2r-3} & \leq  n-(t_r+s_r) \leq \left\lfloor \frac{n-t_r}{2} \right\rfloor \\
    & \ \ \vdots  \\
    2^1 & \leq n-(t_r+s_r+\ldots+t_2+s_2) \leq \left\lfloor \frac{n-(t_r+s_r+\ldots+t_2)}{2} \right\rfloor \\
    2^0 & \leq n-(t_r+s_r+\ldots+t_2+s_2+t_1) \leq \left\lfloor \frac{n-(t_r+s_r+\ldots+t_2+s_2)}{2} \right\rfloor,
\end{align*}
that is, 
\begin{align*}
    \left\lceil \frac{n}{2} \right\rceil & \leq t_r \leq n-2^{2r-2} \\ 
    \left\lceil \frac{n-t_r}{2} \right\rceil & \leq s_r \leq n-t_r-2^{2r-3} \\
    & \vdots \\
    \left\lceil \frac{n-t_2-(s_{3}+t_{3}+\cdots+s_r+t_r)}{2} \right\rceil & \leq s_2 \leq n-t_2-(s_{3}+t_{3}+\cdots+s_r+t_r)-2^1 \\
    \left\lceil \frac{n-(s_2+t_2+\cdots+s_{r}+t_{r})}{2} \right\rceil & \leq t_1 \leq  n-(s_2+t_2+\cdots+s_{r}+t_{r})-2^0\\
     s_1 & = n-t_1-(s_2+t_2+\cdots+s_{r}+t_{r}),
\end{align*}
and the result follows. 

The other side of the prove is clear since Propositions \ref{prop:n} and \ref{prop:n-k}  characterizations of the bounds to each $l_i$, which implies the dimension of each eigenspace, and it is related to number of zeros and ones in the binary sequence. 
\end{proof}


\section{WHD threshold graphs}
\label{sec:4}

Building on the previous characterization of simply structured connected threshold graphs, we now obtain WHD threshold graphs by establishing sufficient conditions for a simply structured threshold graph to be WHD. Before this, we introduce some useful results.

Some families of WHD threshold graphs have already been studied. In \cite{WHD}, Adm \emph{et al.} showed that the threshold graphs $G=K_k^c\vee K_{n-k}$ are WHD under certain conditions on $n$ and $k$. 

\begin{proposition}
\label{heberk>3}
Let $G=K_k^c\vee K_{n}$.  If $n-k \in \{0, 1, 2\}$, then $G$ is WHD.
\end{proposition}

Note that for $k=1$ and $n\geq 2$, then $G \cong K_{n}$ is WHD; if $k=2$ and $n\geq 4$, then $G \cong K_{n} - e$ is WHD.  Macharete et al.\ in \cite[Propositions 4.1 and 4.2]{MDTL24} proved that $K_{k}^{c} \vee K_{n}$ plus an edge is still a WHD graph.

\begin{proposition}\label{heberG+e}
    Let $G=K_k^c\vee K_{n}$ such that $n-k \in \{0, 1, 2\}$ and $k\geq 4$. Then, $G^{\prime} = G+e$ is a WHD threshold graph.
\end{proposition}

In \cite{WeakMatrices}, an infinite family of WHD threshold graphs of order $2^{l}$ for $l\geq 1$ are provided. In Proposition \ref{heberG+e}, we construct an infinite family of WHD threshold graphs that are not described by their work. One such graph is $(K_2 \sqcup K_{6}^c) \vee K_8$. We are also able to generate threshold WHD graphs for any number of vertices, and we are not restricted to threshold graphs of order $n = 2^{l}$ for $l \geq 1.$ These facts make the graphs obtained from Proposition \ref{heberG+e} interesting. 

The result of Proposition \ref{heberG+e} can be generalized as presented in \cite{WHD}. We present an alternative and more precise proof of this result here since it is crucial to prove the main result of this section.  

\begin{proposition}\label{HjoinKn}
Let $G=H \, \vee \, K_{n}$ where $H$ is a WHD connected graph on $k\geq 2$ vertices. If $n-k \in \{0, 1, 2\}$, then $G$ is WHD.
\end{proposition}

\begin{proof}
Let $\spec(H)=\{\mu_\mathbf{1}^{(n_1)}, \ \mu_2^{(n_2)}, \ldots, \ \mu_l^{(n_l)}, \ \mathbf{0}^{(1)}\}$, then 
\begin{align*}
    \spec(G)=\{(n+k)^{(n)}, \ (n+\mu_1)^{(n_1)}, \ (n+\mu_2)^{(n_2)}, \ldots, \ (n+\mu_l)^{(n_l)}, \ \mathbf{0}^{(1)}\}.
    \end{align*}
Note that $\mathbbm{1}_{n+k}$ is an eigenvector of $L(G)$ associated with eigenvalue 0. We arrange the vertices of $G$ so that the first $k$ vertices are the vertices of $H.$ Denote by $u_{i_j}$ the eigenvectors of $L(H)$  associated with eigenvalues $\mu_i$. Since $H$ is WHD, the vectors $u_{i_j}$ form a weak Hadamard matrix which diagonalizes $L(H)$. Moreover, the vectors $u_{i_j}$ concatenated with $n$ zeros, denoted by $U_{i_j}$, form suitable eigenvectors for the eigenspace associated with $n+\mu_i$. 

Similarly, the vectors $\mathbf{v}_i=e_i-e_{i+1}$, with $i=k+1, \ldots, n+k-1$ are $n-1$ suitable eigenvectors of $L(G)$ for the eigenspace associated with $n+k$. The vector $\mathbf{v}=(\underbrace{1,1,\ldots,1}_{k},\underbrace{-1, -1, \ldots, -1}_{n-1},\underbrace{n-k-1}_1)$ is an eigenvector of $L(G)$ associated with eigenvalue $n+k$. Moreover, if $n-k \in\{0, 1, 2\}$, then the eigenvector $\mathbf{v}$ of $L(G)$ has entries from $\{-1, 0, 1\}$. Besides, $\mathbf{v}$ is not a linear combination of the set $\{\mathbf{v}_{k+1},\ldots, \mathbf{v}_{n+k-1}\}$ since the $k$ first entries of each $\mathbf{v}_{i}$ is equal to zero. 

If we form the matrix $W$ where the first $k-1$ columns are the eigenvectors $U_{i_j}$ of $L(G)$ associated with eigenvalues $n+\mu_i$, the next $n$ columns are the eigenvectors  $\mathbf{v}_{k+1}, \ldots, \mathbf{v}_{n+k-1}, \mathbf{v}$ of $L(G)$ associated with eigenvalues $n+k$, and the last column is given by the eigenvector $\mathbbm{1}_{n+k}$ of $L(G)$ associated with eigenvalue 0, then $W^TW$ is tridiagonal, and the proof is complete. 
\end{proof} 

\begin{remark}
\label{remark1}
The graph $G = K_{3}^{c} \vee K_{9}$ is simply structured but does not satisfy the conditions of Proposition \ref{heberk>3}. However, $G$ is WHD since we can rewrite it as a sequence of joins: 
$G = H \vee K_6$, where $H=K_3^{c} \vee K_{3}.$ From Proposition \ref{heberk>3}, $H$ is WHD, and consequently $G$ is WHD by Proposition \ref{HjoinKn}.
\end{remark}

The idea in Remark \ref{remark1} can be generalized. Suppose that $G$ has binary sequence $\mathbf{b}=\mathbf{0}^{s_1} \mathbf{1}^{t_1}.$ If $t_1 - s_1 \in \{0,1,2\}$, then $G$ is WHD, from Proposition \ref{heberk>3}. Now, consider that $t_1 - s_1 \geq 3.$ In this case, we can rewrite $G$ as:
\begin{eqnarray}
    K_{s_1}^{c} \vee K_{t_1} &=& K_{s_1}^{c} \vee K_{s_1 + i} \vee K_{t_1-s_1-i}; \nonumber  \\
    &=& K_{s_1}^{c} \vee K_{s_1 + i} \vee K_{2s_1 + i + j} \vee K_{t_1-3s_1-2i-j}; \nonumber \\
    &=& K_{s_1}^{c} \vee K_{s_1 + i} \vee K_{2s_1 + i + j} \vee K_{4s_1+2_i+j+k} \vee K_{t_1-7s_1-4i-2j-k}. \nonumber 
\end{eqnarray}
Note that if 
$$t_1-7s_1-4i-2j-k - (4s_1+2_i+j+k) - (2s_1 + i + j) - (s_1 + i) - s_1 = p$$
where $ 0\le i, j, k, p \le 2,$ then $G$ is WHD from Proposition \ref{HjoinKn}. It implies that 
$$t_1-15s_1 = 8i+4j+2k+p,$$ thus $t_1 - 15 s_1 \in \{0, \ldots, 30\}.$
By applying the ideas above, we generalize this procedure and obtain that if there exists 
$$t_1-(2^l-1)s_1 = m,$$
where $l \in \{1, \ldots, \lfloor\log_2(\frac{n}{s_1})\rfloor \}$ and $m \in \{0,\ldots,2(2^l-1)\},$ then $G$ is WHD. Note that $l$ is the number of times the complete graph $K_{t_1}$ is represented by a sequence of joins, and the upper bound for $l$ is obtained by assuming $m=0.$ 

The next theorem provides sufficient conditions such that a simply structured threshold graph of a given order is WHD.

\begin{theorem}\label{teorema whd}
    Let $G$ be a threshold graph with $\mathbf{b}=\mathbf{0}^{s_1}\mathbf{1}^{t_1}\cdots\mathbf{0}^{s_r}\mathbf{1}^{t_r}$ and let $k\in \{1,\ldots,r\}$. {If $G$ is simply structured and all of the following conditions hold then $G$ is WHD.}
    \begin{enumerate}
        \item[(1) ] for $k=1$ and 
            \begin{itemize}
                \item[(1.1) ] $s_1=1;$ or
                \item[(1.2) ] $s_1=2;$ or
                \item[(1.3) ] $s_1 \ge 3,$ there exists $l \in \{1, \ldots, \lfloor\log_2(\frac{n}{s_1})\rfloor\}$ and $m \in \{0,\ldots,2(2^l-1)\}$ such that $t_1-(2^l-1)s_1 = m;$
            \end{itemize}
        \item[(2) ] for each $k \in \{2,\ldots,r\}$  there exists $l \in \{1, \ldots, \lfloor\log_2(\frac{n}{s_1})\rfloor - 2^{k-1}\}$ and $m \in \{0,\ldots,2(2^l-1)\}$ such that 
        \begin{align}
            t_k - (2^l-1) \left(\sum_{q=1}^{k-1} (s_q + t_q) + s_k \right)  & = m \label{ti}  \quad \text{and} \\
            s_k - (2^l-1) \left(\sum_{q=1}^{k-1} (s_q + t_q) \right)  & = m \label{si}, 
\end{align}   
    \end{enumerate}
\end{theorem}
\begin{proof}
If $r=1$, then in cases (1.1) and (1.2), the graph $G$ is isomorphic to either $K_{n}$ or $K_{n} - e$, respectively, and both are WHD from Proposition \ref{heberk>3}. Now suppose that (1.3) holds. From Remark \ref{remark1}, $G$ is WHD.

Now assume that $r \geq 2$ and let $k \in \{1,\ldots,r\}.$  The graph $G$ is given by binary sequence $\mathbf{b}=\mathbf{0}^{s_1}\mathbf{1}^{t_1}\cdots\mathbf{0}^{s_r}\mathbf{1}^{t_r}$. Let $G^{\prime}$ be the threshold graph with $\mathbf{b}_{G^{\prime}} = \mathbf{0}^{s_1}\mathbf{1}^{t_1}\mathbf{1}^{s_2}\mathbf{1}^{t_2}\cdots\mathbf{1}^{s_r}\mathbf{1}^{t_r}$. From Theorem \ref{TeoPrincipal}, $G$ and $G^{\prime}$ share the same Laplacian eigenbasis, and since $G$ is simply structured, graph $G^{\prime}$ is also simply structured. Proving that $G^{\prime}$ is WHD implies that $G$ is WHD.
Write $G^{\prime} = ((((K_{s_1}^{c} \vee K_{t_1}) \vee K_{s_2}) \vee K_{t_2}) \cdots \vee K_{s_r}) \vee K_{t_r}$. By hyphothesis (1), we have that   $G^{\prime}_{1} = K_{s_1}^{c} \vee K_{t_1}$ is WHD. From hypothesis (2) and Proposition \ref{HjoinKn}, 
$G^{\prime}_{2} = G^{\prime}_{1} \vee K_{s_2}$ is also WHD. Applying this same procedure, note that  
$G^{\prime}_{i}$ graphs are WHD for $i = 3,\ldots,2r-2,$ and in particular $G^{\prime} = G^{\prime}_{2r-1} = G^{\prime}_{2r-2} \vee K_{t_r}$ is WHD. Therefore, $G^{\prime}$ is WHD. 
\end{proof}

\section{Pair state transfer on threshold graphs}
\label{sec:5}

Let $G$ be a connected graph on $n$ vertices. The {\sl continuous-time quantum walk} on $G$ with Hamiltonian $L(G)$ is determined by a one-parameter family of unitary matrices given by
\begin{equation*}
U(t)=e^{it L(G)},\quad t\in\mathbb{R}.
\end{equation*}
The matrix $U(t)$ governs the propagation of quantum states in $G$ under Laplacian dynamics. A pure (quantum) state on $G$ is a unit vector in $\mathbb{C}^n$ \cite{godsil2025perfect}. In particular, a pure state represented by a standard basis vector is called a \textit{vertex state}, while that of the form $\frac{1}{\sqrt{2}}(\mathbf{e}_a-\mathbf{e}_b)$ with $a\neq b$ is called a \textit{pair state}. 

For $\mu\in \spec(G)$, $E_{\mu}$ denotes the orthogonal projection matrix onto the eigenspace associated with $\mu$. The \textit{eigenvalue support} of pure state $\mathbf{u}$ is the set
\begin{center}
$\supp_{\mathbf{u}}=\{\mu\in \spec(G):E_\mu\mathbf{u}\neq\mathbf{0}\}$. 
\end{center}
For a Laplacian integral graph, it is immediate that $\supp_{\mathbf{u}}$ is an integral set for any pure state $\mathbf{u}$. We also note that if $\mathbf{u}$ is a pair state, then $0\notin \supp_{\mathbf{u}}$. 

Let $\mathbf{u},\mathbf{v}\in\mathbb{R}^n$ be real pure states on $G$ that are linearly independent. We say that {\sl perfect state transfer} (PST for short) occurs between $\mathbf{u}$ and $\mathbf{v}$ at time $\tau>0$ if
there exists $\eta\in\mathbb{C}$ called {\sl phase factor}, such that
\begin{equation}
\label{pstper}
U(\tau) \mathbf{u} = \eta \mathbf{v}.
\end{equation}
The minimum such $\tau$ is called the \textit{minimum PST time}. If $\mathbf{u}$ and $\mathbf{v}$ in (\ref{pstper}) are linearly dependent, then we say that $\mathbf{u}$ is \textit{periodic}. We say that $\mathbf{u}$ and $\mathbf{v}$ are \textit{strongly cospectral} if $E_\mu\mathbf{u}=E_\mu\mathbf{v}$ for each $\mu\in\supp_{\mathbf{u}}$, in which case, we may partition $\supp_{\mathbf{u}}$ into two sets given by
\begin{center}
$\supp_{\mathbf{u},\mathbf{v}}^+=\{\mu\in \supp_{\mathbf{u}}:E_\mu\mathbf{u}=E_\mu\mathbf{v}\}\quad $ and $\quad \supp_{\mathbf{u},\mathbf{v}}^+=\{\mu\in \supp_{\mathbf{u}}:E_\mu\mathbf{u}=-E_\mu\mathbf{v}\}$.
\end{center}

\begin{remark}
\label{remSC}
If $\mathbf{w}$ is an eigenvector for $L(G)$ associated with $\mu\in\supp_{\mathbf{u}}$, then we have $E_\mu\mathbf{w}=\mathbf{w}$, and so $\mathbf{w}^TE_\mu\mathbf{u}=\mathbf{w}^T\mathbf{u}$. Consequently, $E_\mu\mathbf{u}=\pm E_\mu\mathbf{v}$ holds if and only if $\mathbf{w}^T\mathbf{u}=\pm \mathbf{w}^T\mathbf{v}\neq 0$ for every eigenvector $\mathbf{w}$ associated with $\mu$.
\end{remark}

A pure state $\mathbf{u}$ is \textit{fixed} if $|\supp_{\mathbf{u}}|=1$, and this happens if and only if $\mathbf{u}$ is an eigenvector for $L(G)$ \cite[Proposition 2.3]{godsil2025perfect}. The involvement of $\mathbf{u}$ in strong cospectrality requires that $|\supp_{\mathbf{u}}|\geq 2$ \cite[Section 4]{godsil2025perfect}. Thus fixed states are not involved in strong cospectrality (and by extension, PST). In particular, the pair state $\frac{1}{\sqrt{2}}(\mathbf{e}_a-\mathbf{e}_b)$ is fixed whenever $a$ and $b$ are twins by \cite[Lemma 2.9]{Monterde2022}. 

Most work on PST deals with vertex states, also called vertex state transfer, see \cite{kirkland2026quantum,MONTERDE2025} for recent work. 
However, due to the rarity of vertex state transfer \cite{godsil2012state}, there is a growing body of work on PST involving other forms of pure states \cite{chen2020pair,godsil2025quantum,kim2024generalization}. In this section, we characterize PST between pair states, also called \textit{pair state transfer}, on threshold graphs. To do this, we first provide a characterization of pair state transfer on Laplacian integral graphs. In what follows, $\nu_2(m)$ denotes the the largest power of two dividing an integer $m$.

\begin{theorem}
\label{pst}
Let $\mathbf{u}=\frac{1}{\sqrt2}(\mathbf{e}_a-\mathbf{e}_b)$ and $\mathbf{v}=\frac{1}{\sqrt2}(\mathbf{e}_c-\mathbf{e}_d)$ be pair states in $G$.
\begin{enumerate}
\item If $|\supp_{\mathbf{u}}|=2$, then perfect state transfer occurs between $\mathbf{u}$ and $\mathbf{v}$ if and only if they are strongly cospectral. 
\item If $G$ is Laplacian integral and $|\supp_{\mathbf{u}}|\geq 3$, then perfect state transfer occurs between $\mathbf{u}$ and $\mathbf{v}$ if and only if both conditions below hold.
\begin{enumerate}
\item The pair states $\mathbf{u}$ and $\mathbf{v}$ are strongly cospectral.
\item For all $\sigma,\mu\in\supp_{\mathbf{u},\mathbf{v}}^+$ and $\lambda,\gamma\in\supp_{\mathbf{u},\mathbf{v}}^-$, we have
\begin{equation*}
\nu_2(\sigma-\mu)>\nu_2(\lambda-\mu)=\nu_2(\gamma-\mu).
\end{equation*}
\end{enumerate}
\end{enumerate}
In both cases, for a fixed $\mu\in\supp_{\mathbf{u},\mathbf{v}}^+$, the minimum PST time is $\tau=\frac{\pi}{g}$ where $g=\operatorname{gcd}\{\mu-\alpha:\alpha\in \supp_{\mathbf{u}}\}$.
\end{theorem}

\begin{proof}
Statements 1 and 2 are immediate from Theorem 5.2(1) and Corollary 5.7 in \cite{godsil2025perfect}, respectively.
\end{proof}

Theorem \ref{pst} applies to threshold graphs since they are Laplacian integral. In order to characterize pair state transfer in this family, we first characterize strong cospectrality between pair states using Theorem \ref{pst}(1-2). From (\ref{th1})-(\ref{th2}), a connected threshold graph $G$ with $\mathbf{0}^{s_1}\mathbf{1}^{t_1}\cdots \mathbf{0}^{s_r}\mathbf{1}^{t_r}$ may be written as
\begin{equation}
\label{th3}
  G=\begin{cases}
    (((((K_{t_1'}\sqcup K^c_{s_2})\vee K_{t_2})\sqcup K^c_{s_3})\cdots)\sqcup K^c_{s_r})\vee K_{t_{r}} & \text{if $s_1=1$},\\
    (((((K^c_{s_1}\vee K_{t_1})\sqcup K^c_{s_2})\vee K_{t_2})\cdots)\sqcup K^c_{s_r})\vee K_{t_{r}} & \text{if $s_1\geq 2$}
  \end{cases}
\end{equation}
where $t_1'=t_1+1$.

\begin{lema}
\label{sc}
Let $G$ be a connected threshold graph with $\mathbf{b}=\mathbf{0}^{s_1}\mathbf{1}^{t_1}\cdots \mathbf{0}^{s_r}\mathbf{1}^{t_r}$. The pair states $\mathbf{u}$ and $\mathbf{v}$ are strongly cospectral in $G$
if and only if the following conditions all hold.
\begin{enumerate}
\item $\mathbf{u}=\frac{1}{\sqrt2}(\mathbf{e}_1-\mathbf{e}_b)$ and $\mathbf{v}=\frac{1}{\sqrt2}(\mathbf{e}_2-\mathbf{e}_b)$ for any $b\geq 3$.
\item Either (i) $s_1=t_1=1$, or (ii) $s_1=2$.
\end{enumerate}
Moreover, if $\mathbf{u}$ and $\mathbf{v}$ are strongly cospectral, then $\supp_{\mathbf{u},\mathbf{v}}^+=\spec(G)\backslash\{0,\theta\} $ and $ \supp_{\mathbf{u},\mathbf{v}}^-=\{\theta\}$, where
\begin{equation}
\label{theta}
\theta=
  \begin{cases}
    2+\sum_{j=2}^rt_j & \text{if $s_1=t_1=1$},\\
    \sum_{j=2}^rt_j & \text{if $s_1=2$}.
  \end{cases}
\end{equation}
\end{lema}

\begin{proof}
Let $G$ be a connected threshold graph. Assume $\mathbf{u}=\frac{1}{\sqrt2}(\mathbf{e}_a-\mathbf{e}_b)$ and $\mathbf{v}=\frac{1}{\sqrt2}(\mathbf{e}_c-\mathbf{e}_d)$ are strongly cospectral in $G$. By Theorem \ref{TeoPrincipal}, $\{\mathbf{x}^{1}, \ldots, \mathbf{x}^{n}\}$ is an eigenbasis for $L(G)$. We divide the proof into two cases. First, suppose $b\neq d$. Without loss of generality, let $b<d$ so that $d\geq 2$. Then 
\[
  (\mathbf{x}^{d})^T\mathbf{u}= \frac{1}{\sqrt{2}}(\mathbf{x}^{d}_a-\mathbf{x}^{d}_b)=
  \begin{cases}
    \frac{1}{\sqrt{2}}(1-1)=0 & \text{if $a<d$}\\
    \frac{1}{\sqrt2}(0-1) & \text{if $a>d$}
  \end{cases}
\]
while 
\[
  (\mathbf{x}^{d})^T\mathbf{v}= \frac{1}{\sqrt{2}}(\mathbf{x}^{d}_c-\mathbf{x}^{d}_d)=
  \begin{cases}
    \frac{1}{\sqrt{2}}(1-(-d)) & \text{if $c<d$}\\
    \frac{1}{\sqrt2}(0-(-d)) & \text{if $c>d$}.
  \end{cases}
\]
Thus, if $a\neq d$, then $(\mathbf{x}^{d})^T\mathbf{u}\neq\pm (\mathbf{x}^{d})^T\mathbf{v}$. Making use of Remark \ref{remSC}, we get that $\mathbf{u}$ and $\mathbf{v}$ are not strongly cospectral, a contradiction. 
If $a=d$, then $\mathbf{u}=-\frac{1}{\sqrt2}(\mathbf{e}_b-\mathbf{e}_d)$ and $\mathbf{v}=\frac{1}{\sqrt2}(\mathbf{e}_c-\mathbf{e}_d)$ which will be covered by the next case.

Next, let $b=d$ so that $\mathbf{u}=\frac{1}{\sqrt2}(\mathbf{e}_a-\mathbf{e}_b)$ and $\mathbf{v}=\frac{1}{\sqrt2}(\mathbf{e}_c-\mathbf{e}_b)$ with $a\neq c$. 
If $a\geq 3$, then $(\mathbf{x}^{a})^T\mathbf{v}= \frac{1}{\sqrt{2}}(\mathbf{x}^{a}_c-\mathbf{x}^{a}_b)\in\{0,\pm 1\}$ whenever $a\neq b,c$, while
\[
  (\mathbf{x}^{a})^T\mathbf{u}= \frac{1}{\sqrt{2}}(\mathbf{x}^{a}_a-\mathbf{x}^{a}_b)=
  \begin{cases}
    \frac{1}{\sqrt{2}}(-a-1) & \text{if $a<b$}\\
    \frac{1}{\sqrt2}(-a-0) & \text{if $a>b$}.
  \end{cases}
\]
In any case, we get $(\mathbf{x}^{a})^T\mathbf{u}\neq\pm (\mathbf{x}^{a})^T\mathbf{v}$. Thus, $\mathbf{u}$ and $\mathbf{v}$ are not strongly cospectral, a contradiction. Therefore, $a\leq 2$, and by symmetry, $c\leq 2$. This implies that $a=1$ and $c=2$, and so $\mathbf{u}=\frac{1}{\sqrt2}(\mathbf{e}_1-\mathbf{e}_b)$ and $\mathbf{v}=\frac{1}{\sqrt2}(\mathbf{e}_2-\mathbf{e}_b)$ for any $b\geq 3$. This proves that 1 is a necessary condition for strong cospectrality between $\mathbf{u}$ and $\mathbf{v}$. To prove the above result, we now show that if $\mathbf{u}=\frac{1}{\sqrt2}(\mathbf{e}_1-\mathbf{e}_b)$ and $\mathbf{v}=\frac{1}{\sqrt2}(\mathbf{e}_2-\mathbf{e}_b)$, then $\mathbf{u}$ and $\mathbf{v}$ are strongly cospectral for any $b\geq 3$ if and only if 2 holds. In this case, $\theta\in\spec(G)$ where
\[
  \theta=
  \begin{cases}
    1+t_1+\sum_{j=2}^rt_j & \text{if $s_1=1$},\\
    \sum_{j=2}^rt_j & \text{if $s_1\geq 2$}
  \end{cases}
\]
with multiplicity $m_\theta=t_1$ if $s_1=1$ and $m_\theta=s_1-1$ otherwise. Since $\{1,\ldots,t_1+1\}$ and $\{1,\ldots,s_1\}$ are sets of pairwise twins whenever $s_1=1$ and $s_1\geq 2$, the vectors $\{\mathbf{x}^{1},\ldots,\mathbf{x}^{t_1}\}$ and $\{\mathbf{x}^{1},\ldots,\mathbf{x}^{s_1-1}\}$ form a basis of eigenvectors for the eigenspace of $\theta$, respectively. Thus, if $A\oplus B$ denotes the direct sum of matrices $A$ and $B$, then orthonormalizing these two bases yields
\[
  E_\theta=
  \begin{cases}
    \big(I_{t_1+1}-\frac{1}{t_1+1}J_{t_1+1}\big)\oplus O & \text{if $s_1=1$},\\
    \big(I_{s_1}-\frac{1}{s_1}J_{s_1}\big)\oplus O & \text{if $s_1\geq 2$},
  \end{cases}
\]
where $J_m$ is the all-ones square matrix of order $m$. Now, suppose $s_1=1$. Then 
\begin{center}
$E_\theta\mathbf{u}=\frac{1}{\sqrt2}\big(\mathbf{e}_1-\frac{1}{t_1+1}\mathbbm{1}_{t_1+1}\big)\quad $ and $\quad E_\theta\mathbf{v}=\frac{1}{\sqrt2}\big(\mathbf{e}_2-\frac{1}{t_1+1}\mathbbm{1}_{t_1+1}\big)$.
\end{center}
From these equations, we see that $E_\theta\mathbf{u}\neq E_\theta\mathbf{v}$. Meanwhile, $E_\theta\mathbf{u}=- E_\theta\mathbf{v}$ if and only if $t_1=1$. In this case $\theta=2+\sum_{j=2}^rt_j$ is a simple eigenvalue of $L(G)$ with associated eigenvector $\mathbf{x}^{1}$, and $E_\mu\mathbf{u}=E_\mu\mathbf{v}$ for all $\mu\in\spec(G)\backslash\{0,\theta\} $. This proves that $\supp_{\mathbf{u},\mathbf{v}}^+=\spec(G)\backslash\{0,\theta\} $ and $ \supp_{\mathbf{u},\mathbf{v}}^-=\{\theta\}$. Similarly, if $s_1\geq 2$, then $E_\theta\mathbf{u}=- E_\theta\mathbf{v}$ if and only if $s_1=2$. In this case, $\theta=\sum_{j=2}^rt_j$, $\supp_{\mathbf{u},\mathbf{v}}^+=\spec(G)\backslash\{0,\theta\}$ and $ \supp_{\mathbf{u},\mathbf{v}}^-=\{\theta\}$. Thus, $\mathbf{u}=\frac{1}{\sqrt2}(\mathbf{e}_1-\mathbf{e}_b)$ and $\mathbf{v}=\frac{1}{\sqrt2}(\mathbf{e}_2-\mathbf{e}_b)$ are strongly cospectral for any $b\geq 3$ if and only if 2 holds, in which case $\supp_{\mathbf{u},\mathbf{v}}^+=\spec(G)\backslash\{0,\theta\} $ and $ \supp_{\mathbf{u},\mathbf{v}}^-=\{\theta\}$.
\end{proof}

\begin{remark}
\label{r1}
It is evident from Lemma \ref{specG} that a connected threshold graph with $\mathbf{b}=\mathbf{0}^{s_1}\mathbf{1}^{t_1}\cdots \mathbf{0}^{s_r}\mathbf{1}^{t_r}$ has exactly $2r$ distinct eigenvalues whenever $s_1=1$ and $2r+1$ whenever $s_1\geq 2$. Let $\mathbf{u}=\frac{1}{\sqrt2}(\mathbf{e}_1-\mathbf{e}_b)$ and $\mathbf{v}=\frac{1}{\sqrt2}(\mathbf{e}_2-\mathbf{e}_b)$ for some $b\geq 3$. Since 0 does not belong to the eigenvalue support of a pair state, the only time get $|\supp_{\mathbf{u}}|=2$ is when 
$r=1$ and $s_1\geq 2$ (recall that $\mathbf{u}$ is fixed whenever $r=1$ and $s_1=1$). In this case, Theorem \ref{pst}(1) implies that $\mathbf{u}$ and $\mathbf{v}$ admit PST in $G$ if and only if the two conditions in Lemma \ref{sc} hold. 
\end{remark}

For the case when $|\supp_{\mathbf{u}}|\geq 3$ (i.e., $r\geq 2$), we have the following result.

\begin{theorem}
\label{pairthresh}
Let $G$ be a connected threshold graph with $\mathbf{b}=\mathbf{0}^{s_1}\mathbf{1}^{t_1}\cdots \mathbf{0}^{s_r}\mathbf{1}^{t_r}$ where $r\geq 2$. Then pair state transfer occurs between $\mathbf{u}$ and $\mathbf{v}$ in $G$ if and only if the following conditions all hold.
\begin{enumerate}
\item $\mathbf{u}=\frac{1}{\sqrt2}(\mathbf{e}_1-\mathbf{e}_b)$ and $\mathbf{v}=\frac{1}{\sqrt2}(\mathbf{e}_2-\mathbf{e}_b)$ for any $b\geq 3$.
\item One of the following conditions is satisfied.
\begin{enumerate}
\item $s_1=t_1=1$, $s_2\equiv 2$ (mod 4), $s_j\equiv 0$ (mod 4) for all $j\geq 3$, and $t_j\equiv 0$ (mod 4) for all $j\geq 2$.
\item $s_1=2$, then $t_1\equiv 2$ (mod 4), and $s_j,t_j\equiv 0$ (mod 4) for all $j\geq 2$.
\end{enumerate}
\end{enumerate}
\end{theorem}

\begin{proof}
By virtue of Theorem \ref{pst}(2) and Lemma \ref{sc}, condition 1 above together with either (i) $s_1=t_1=1$ or (ii) $s_1=2$ are necessary conditions for PST. In this case, $\mathbf{u}=\frac{1}{\sqrt2}(\mathbf{e}_1-\mathbf{e}_b)$ and $\mathbf{v}=\frac{1}{\sqrt2}(\mathbf{e}_2-\mathbf{e}_b)$ for any $b\geq 3$ are strongly cospectral in $G$ with $\supp_{\mathbf{u},\mathbf{v}}^+=\spec(G)\backslash\{0,\theta\} $ and $ \supp_{\mathbf{u},\mathbf{v}}^-=\{\theta\}$, where $\theta$ is given in (\ref{theta}). To prove the above result, it suffices to show that under the preceding conditions, $\mathbf{u}$ and $\mathbf{v}$ admit PST in $G$ if and only if either condition 2a or 2b above holds, or equivalently, the $\nu_2(\mu-\theta)$'s are equal for all $\mu\in\supp_{\mathbf{u},\mathbf{v}}^+$ by Theorem \ref{pst}(2). For the case $s_1=t_1=1$, Lemma \ref{specG} gives us
\begin{align*}
    \spec(G) = \big\{ 
    & n, n-s_r, \cdots, n-(s_2+\cdots+s_r), t_2+\cdots+t_r, \cdots, t_{r-1}+t_r, t_r, 0 \big\}
\end{align*}
where $\theta=2+\sum_{j=2}^rt_j=n-(s_2+\cdots+s_r)$. Hence, if $\mu\in \spec(G)\backslash\{0\}$, then
\begin{align*}
    \mu-\theta\in \big\{ 
    & s_2+\cdots+s_r, \ s_2+\cdots+s_{r-1}, \ \ldots \ , \ s_2, \\
    &t_2+\cdots+t_r+s_2+\cdots+s_r-n, \ \ldots \ , \ t_{r-1}+t_r+s_2+\cdots+s_r-n, \\
    &t_r+s_2+\cdots+s_r-n\big\}.
\end{align*}
Now, observe that $\nu_2(s_2)=\nu_2(s_2+\sum_{j=3}^\ell s_j)=\nu_2(\sum_{j=2}^rs_j)$ for each $\ell\in\{3,\ldots,r\}$ if and only if $s_2\equiv 2$ (mod 4) and $s_j\equiv 0$ (mod 4) for $j\geq 3$. In this case, $\nu_2(s_2)=1$, and we obtain $\nu_2(\sum_{j=\ell}^rt_j+\sum_{j=2}^rs_j-n)=1$ for each $\ell\in\{2,\ldots,r\}$ if and only if $t_j\equiv 0$ (mod 4) for $j\geq 2$ (so that $n\equiv 0$ (mod 4)). Therefore, the $\nu_2(\mu-\theta)$'s are equal for all $\mu\in\supp_{\mathbf{u},\mathbf{v}}^+$ if and only if condition 2a holds. The conclusion for the case $s_1=2$ can be derived similarly.
\end{proof}

\begin{corollary}
\label{corlast}
Let $G$ be a connected threshold graph with  $\mathbf{b}=\mathbf{0}^{s_1}\mathbf{1}^{t_1}\cdots \mathbf{0}^{s_r}\mathbf{1}^{t_r}$. 
\begin{enumerate}
\item Suppose $r=s_1=1$. Then $G=K_n$ does not admit pair state transfer and vertex state transfer.
\item Suppose $r=1$ and $s_1\geq 2$. Then $G=K^c_{s_1}\vee K_{t_1}$ admits pair state transfer if and only if $s_1=2$, in which case, it occurs between $\mathbf{u}=\frac{1}{\sqrt2}(\mathbf{e}_1-\mathbf{e}_b)$ and $\mathbf{v}=\frac{1}{\sqrt2}(\mathbf{e}_2-\mathbf{e}_b)$ for any $b\in V(K_{t_1})$. Moreover, vertex state transfer occurs in $G$ between vertices $a$ and $b$ if and only if $a=1$, $b=2$, $s=2$ and $t_1\equiv 2$ (mod 4). In both cases, the minimum PST time is $\frac{\pi}{2}$.
\item Suppose $r\geq 2$. Then $G$ admits pair state transfer if and only if it admits vertex state transfer at the same time. In particular, $\mathbf{u}=\frac{1}{\sqrt2}(\mathbf{e}_1-\mathbf{e}_b)$ and $\mathbf{v}=\frac{1}{\sqrt2}(\mathbf{e}_2-\mathbf{e}_b)$ admit PST at time $\tau$ in $G$ for any $b\geq 3$ if and only if vertices 1 and 2 admit PST at time $\tau$ in $G$. In this case, the minimum PST time is $\tau=\frac{\pi}{2}$ and vertex $b$ is periodic in $G$ at time $\tau$.
\end{enumerate}
\end{corollary}

\begin{proof}
Note that 1 follows from the facts that each pair state in $K_n$ is fixed and it does not admit vertex PST.
Now, let $r=1$ and $s_1\geq 2$ so that $G=K^c_2\vee K_{t_1}$. By Remark \ref{sc}, pair state transfer occurs in $G$ if and only if $\mathbf{u}=\frac{1}{\sqrt2}(\mathbf{e}_1-\mathbf{e}_b)$ and $\mathbf{v}=\frac{1}{\sqrt2}(\mathbf{e}_2-\mathbf{e}_b)$ for any $b\in V(K_{t_1})$ and $s_1=2$, with minimum PST time $\frac{\pi}{2}$. Meanwhile, PST occurs between vertices $a$ and $b$ in $G$ if and only if $a,b\in\{1,2\}$, $s_1=2$ and $t_1\equiv 2$ (mod 4), in which case the minimum PST time is $\frac{\pi}{2}$ \cite[Corollary 4]{alvir2016perfect}. This proves 2, while 3 is immediate from combining Theorem \ref{pairthresh} and  \cite[Theorem 2]{kirkland2011spin}.
\end{proof}

The next result is an immediate consequence of Corollary \ref{corlast}.

\begin{corollary}
If $r=1$, $s_1=2$ and $t_1\not\equiv 2$ (mod 4), then pair state transfer occurs in a connected threshold graph but not vertex state transfer. Otherwise, the two quantum phenomena are equivalent in a connected threshold graph.
\end{corollary}

From the above corollary, it follows that for each integer $n\geq 3$, vertex state transfer on a connected threshold graph on $n$ vertices implies pair state transfer. Consequently, for each integer $n\geq 3$, there are more connected threshold graphs on $n$ vertices admitting pair state transfer than vertex state transfer. In particular,  $K_2^c\vee K_{n-2}$ for $n\not\equiv 2$ (mod 4) is an infinite family of connected threshold graphs admitting pair state transfer but not vertex state transfer. 

\section{Open questions}
\label{sec:6}

In this work, we characterized simply structured threshold graphs and gave sufficient conditions for them to be WHD. We also characterized Laplacian pair state transfer on threshold graphs. It is worthwhile to determine all threshold graphs (and in general, cographs) that are WHD. We also seek a characterization of cographs that admit either vertex state transfer or pair state transfer. Finally, to obtain more examples, it is desirable to enumerate WHD graphs of small order and to develop further tools to determine WHD graphs of any given order, similar to the work in \cite{breen}. This is particularly beneficial when the weak Hadamard matrix in question has pairwise orthogonal columns, since it is easier to determine the existence of vertex state transfer in such a WHD graph \cite{WeakMatrices}.

\bigskip

\noindent \textbf{Acknowledgement.}
H.~Monterde is supported by the Pacific Institute for the Mathematical Sciences through the PIMS-Simons Postdoctoral Fellowship. L.~de Lima is supported by CNPq grants 305988/2025-5 and 403963/2021-4.

\bigskip

\noindent \textbf{Declaration of competing interest.} None.

\appendix 

\section{List of all threshold graphs with structured eigenbases on at most 20 vertices}

In Table \ref{tab:tab1}, we list all connected threshold graphs on at most 20 vertices with simply structured eigenbases. The columns of this table constitute the order $n$, the binary sequence $\mathbf{b}$, the threshold graph written as a sequence of complements and joins, whether they are WHD (WHD column), and whether they admit pair state transfer (PST column). Note that the blank cells under the PST column mean that the corresponding simply structured graphs do not admit pair state transfer. Meanwhile, the blank cells under the WHD column mean that we do not know yet whether the corresponding simply structured graphs are WHD. Nonetheless, our numerical observations suggest that all simply structured graphs with blank cells under the WHD column are not WHD.

\begin{center}
\begin{longtable}{p{0.7cm} p{2cm} p{4cm} p{1cm} p{0cm}}
\caption{Table of simply structured and WHD threshold graphs}
\label{tab:tab1}\\

\hline \multicolumn{1}{l}{$n$} & \multicolumn{1}{l}{\hspace{-0.15cm} Binary seq} & \multicolumn{1}{l}{Threshold graph}  & \multicolumn{1}{l}{\hspace{-0.18cm} WHD} & \multicolumn{1}{l}{\hspace{-0.1cm} PST} \\ \hline \endfirsthead

\multicolumn{5}{c}%
{{\bfseries \tablename\ \thetable{} -- (continued)...}} \\ \hline

\multicolumn{1}{l}{$n$} & \multicolumn{1}{l}{\hspace{-0.15cm} Binary seq} & \multicolumn{1}{l}{Threshold graph}  & \multicolumn{1}{l}{\hspace{-0.18cm} WHD} & \multicolumn{1}{l}{\hspace{-0.1cm} PST} \\ 
\hline 
\endhead

\hline \multicolumn{5}{r}{{(continued)...}} \\ \hline
\endfoot
\endlastfoot
3 & $\mathbf{0}\mathbf{1}^{2}$ & $K_{3}$ & Yes & \\
\hline
4 & $\mathbf{0}\mathbf{1}^{3}$ & $ K_{4}$ & Yes &  \\
  & $\mathbf{0}^{2}\mathbf{1}^{2}$ & $K_{2}^c \vee K_{2}$ &  Yes & Yes \\
\hline
5 &  $\mathbf{0}\mathbf{1}^{4}$ & $K_5$ &  & \\
  &  $\mathbf{0}^{2}\mathbf{1}^{3}$ & $ K_{2}^c \vee K_{3}$ &  Yes & Yes\\
  \hline
6 & $\mathbf{0}\mathbf{1}^{5}$ & $K_6$ &  Yes & \\
& $\mathbf{0}^{2}\mathbf{1}^{4}$ & $K_{2}^c \vee K_{4}$ &  Yes & Yes \\
& $\mathbf{0}^{3}\mathbf{1}^{3}$ & $K_{3}^c \vee K_{3}$ &  Yes & \\
\hline
7 & $\mathbf{0}\mathbf{1}^{6}$ & $ K_{7}$ &  Yes & \\
  & $\mathbf{0}^2\mathbf{1}^5$ & $K_{2}^c \vee K_{5}$  &  Yes & Yes\\
  & $\mathbf{0}^3\mathbf{1}^3$ & $K_{3}^c \vee K_{4}$ &  Yes &\\
  \hline
8 & $\mathbf{0}\mathbf{1}^7$ & $ K_{8}$ &  Yes& \\
&$\mathbf{0}^{2}\mathbf{1}^6$ & $ K_{2}^c \vee K_{6}$ &  Yes & Yes\\
&$\mathbf{0}^3\mathbf{1}^5$ & $ K_{3}^c \vee K_{5}$ &  Yes&\\
&$\mathbf{0}^4 \mathbf{1}^4$ & $ K_{4}^c \vee K_{4}$ &  Yes&\\
&$\mathbf{0}\mathbf{1}\mathbf{0}^2\mathbf{1}^4$ & $ (K_{2} \sqcup K_{2}^c) \vee K_{4}$ &  Yes& \\
\hline
9& $\mathbf{0}\mathbf{1}^8$ & $K_{9}$ &  Yes& \\
& $\mathbf{0}^2\mathbf{1}^7$ & $ K_{2}^c \vee K_{7}$ &  Yes& \\
& $\mathbf{0}^3\mathbf{1}^6$ & $ K_{3}^c \vee K_{6}$ & &\\
& $\mathbf{0}^4\mathbf{1}^5$ & $ K_{4}^c \vee K_{5}$ &  Yes&\\
& $\mathbf{0}\mathbf{1}\mathbf{0}^2\mathbf{1}^{5}$ & $(K_{2} \sqcup K_{2}^c )\vee K_{5}$ & & Yes \\
\hline
10&$\mathbf{0}\mathbf{1}^9$ & $ K_{10}$ & Yes& \\
&$\mathbf{0}^2 \mathbf{1}^8$ & $ K_{2}^c \vee K_{8}$ & Yes& Yes \\
&$\mathbf{0}^3 \mathbf{1}^7$ & $ K_{3}^c \vee K_{7}$ & &\\
&$\mathbf{0}^4 \mathbf{1}^6$ & $ K_{4}^c \vee K_{6}$ & Yes&\\
&$\mathbf{0}^5 \mathbf{1}^5$ & $ K_{5}^c \vee K_{5}$ & Yes&\\
&$\mathbf{0}\mathbf{1} \mathbf{0}^2 \mathbf{1}^6$ & $ (K_{2} \sqcup K_{2}^c) \vee K_{6}$ & Yes&\\
&$\mathbf{0}\mathbf{1}\mathbf{0}^3 \mathbf{1}^5$ & $ (K_{2} \sqcup K_{3}^c )\vee K_{5}$ & Yes&\\
\hline
11&$\mathbf{0}\mathbf{1}^{10}$ & $ K_{11}$ & Yes&\\
&$\mathbf{0}^2 \mathbf{1}^9$ & $ K_{2}^c \vee K_{9}$ & Yes&Yes\\
&$\mathbf{0}^3 \mathbf{1}^8$ & $ K_{3}^c \vee K_{8}$ & &\\
&$\mathbf{0}^4 \mathbf{1}^7$ & $ K_{4}^c \vee K_{7}$ & &\\
&$\mathbf{0}^5 \mathbf{1}^6$ & $ K_{5}^c \vee K_{6}$ & Yes&\\
&$\mathbf{0}\mathbf{1}\mathbf{0}^2 \mathbf{1}^7$ & $ (K_{2} \sqcup K_{2}^c )\vee K_{7}$ & &\\
&$\mathbf{0}\mathbf{1}\mathbf{0}^3 \mathbf{1}^6$ & $( K_{2} \sqcup K_{3}^c )\vee K_{6}$ & Yes &\\
\hline
12&$\mathbf{0}\mathbf{1}^{11}$ & $ K_{12}$ & Yes&\\
&$\mathbf{0}^2 \mathbf{1}^{10}$ & $ K_{2}^c \vee K_{10}$ & Yes & Yes \\
&$\mathbf{0}^3 \mathbf{1}^{9}$ & $ K_{3}^c \vee K_{9}$ & Yes &\\
&$\mathbf{0}^4 \mathbf{1}^8$ & $ K_{4}^c \vee K_{8}$ & &\\
&$\mathbf{0}^5 \mathbf{1}^{7}$ & $ K_{5}^c \vee K_{7}$ & Yes&\\
&$\mathbf{0}^6 \mathbf{1}^{6}$ & $ K_{6}^c \vee K_{6}$ & Yes&\\
&$\mathbf{0}\mathbf{1}\mathbf{0}^2 \mathbf{1}^8$ & $ (K_{2} \sqcup K_{2}^c )\vee K_{8}$ & & Yes\\
&$\mathbf{0}\mathbf{1}\mathbf{0}^3 \mathbf{1}^7$ & $ (K_{2} \sqcup K_{3}^c )\vee K_{7}$ & Yes&\\
&$\mathbf{0}\mathbf{1}\mathbf{0}^4 \mathbf{1}^6$ & $ (K_{2} \sqcup K_{4}^c )\vee K_{6}$ & Yes&\\
&$\mathbf{0}\mathbf{1}^2 \mathbf{0}^3 \mathbf{1}^6$ & $ (K_{3} \sqcup K_{3}^c )\vee K_{6}$ & Yes&\\
\hline
13&$\mathbf{0}\mathbf{1}^{12}$ & $ K_{13}$ & Yes&\\
&$\mathbf{0}^2 \mathbf{1}^{11}$ & $ K_{2}^c \vee K_{11}$ & Yes&Yes\\
&$\mathbf{0}^3 \mathbf{1}^{10}$ & $ K_{3}^c \vee K_{10}$ & Yes&\\
&$\mathbf{0}^{4} \mathbf{1}^{9}$ & $ K_{4}^c \vee K_{9}$ & &\\
&$\mathbf{0}^{5} \mathbf{1}^{8}$ & $ K_{5}^c \vee K_{8}$ & &\\
&$\mathbf{0}^6 \mathbf{1}^{7}$ & $ K_{6}^c \vee K_{7}$ & Yes&\\
&$\mathbf{0}\mathbf{1}\mathbf{0}^2 \mathbf{1}^{9}$ & $( K_{2} \sqcup K_{2}^c )\vee K_{9}$ & &\\
&$\mathbf{0}\mathbf{1}\mathbf{0}^3 \mathbf{1}^8$ & $ (K_{2} \sqcup K_{3}^c )\vee K_{8}$ & &\\
&$\mathbf{0}\mathbf{1}\mathbf{0}^{4} \mathbf{1}^{7}$ & $( K_{2} \sqcup K_{4}^c )\vee K_{7}$ & Yes&\\
&$\mathbf{0}\mathbf{1}^2 \mathbf{0}^3 \mathbf{1}^{7}$ & $( K_{3} \sqcup K_{3}^c )\vee K_{7}$ & Yes&\\
\hline
14&$\mathbf{0}\mathbf{1}^{13}$ & $ K_{14}$ & Yes&\\
&$\mathbf{0}^2 \mathbf{1}^{12}$ & $ K_{2}^c \vee K_{12}$ & Yes&Yes\\
&$\mathbf{0}^3 \mathbf{1}^{11}$ & $ K_{3}^c \vee K_{11}$ & Yes&\\
&$\mathbf{0}^4 \mathbf{1}^{10}$ & $ K_{4}^c \vee K_{10}$ & &\\
&$\mathbf{0}^5 \mathbf{1}^{9}$ & $ K_{5}^c \vee K_{9}$ & &\\
&$\mathbf{0}^6 \mathbf{1}^8$ & $ K_{6}^c \vee K_{8}$ & Yes&\\
&$\mathbf{0}^7 \mathbf{1}^7$ & $ K_{7}^c \vee K_{7}$ & Yes&\\
&$\mathbf{0}\mathbf{1}\mathbf{0}^2 \mathbf{1}^{10}$ & $ (K_{2} \sqcup K_{2}^c )\vee K_{10}$ & &\\
&$\mathbf{0}\mathbf{1}\mathbf{0}^3 \mathbf{1}^9$ & $ (K_{2} \sqcup K_{3}^c )\vee K_{9}$ & &\\
&$\mathbf{0}\mathbf{1}\mathbf{0}^4 \mathbf{1}^8$ & $ (K_{2} \sqcup K_{4}^c )\vee K_{8}$ & Yes&\\
&$\mathbf{0}\mathbf{1}^2 \mathbf{0}^3 \mathbf{1}^8$ & $( K_{3} \sqcup K_{3}^c )\vee K_{8}$ & Yes&\\
&$\mathbf{0}\mathbf{1}\mathbf{0}^5 \mathbf{1}^7$ & $ (K_{2} \sqcup K_{5}^c )\vee K_{7}$ & &\\
&$\mathbf{0}\mathbf{1}^2 \mathbf{0}^4 \mathbf{1}^7$ & $( K_{3} \sqcup K_{4}^c )\vee K_{7}$ & Yes&\\
\hline
15&$\mathbf{0}\mathbf{1}^{14}$ & $ K_{15}$ & Yes&\\
&$\mathbf{0}^2 \mathbf{1}^{13}$ & $ K_{2}^c \vee K_{13}$ & Yes&Yes\\
&$\mathbf{0}^3 \mathbf{1}^{12}$ & $ K_{3}^c \vee K_{12}$ & Yes&\\
&$\mathbf{0}^4 \mathbf{1}^{11}$ & $ K_{4}^c \vee K_{11}$ & &\\
&$\mathbf{0}^5 \mathbf{1}^{10}$ & $ K_{5}^c \vee K_{10}$ & &\\
&$\mathbf{0}^6 \mathbf{1}^9$ & $ K_{6}^c \vee K_{9}$ & &\\
&$\mathbf{0}^7 \mathbf{1}^8$ & $ K_{7}^c \vee K_{8}$ & Yes&\\
&$\mathbf{0}\mathbf{1}\mathbf{0}^2 \mathbf{1}^{11}$ & $ (K_{2} \sqcup K_{2}^c )\vee K_{11}$ & &\\
&$\mathbf{0}\mathbf{1}\mathbf{0}^3 \mathbf{1}^{10}$ & $ (K_{2} \sqcup K_{3}^c )\vee K_{10}$ & &\\
&$\mathbf{0}\mathbf{1}\mathbf{0}^4 \mathbf{1}^{9}$ & $ (K_{2} \sqcup K_{4}^c )\vee K_{9}$ & &\\
&$\mathbf{0}\mathbf{1}^2 \mathbf{0}^3 \mathbf{1}^{9}$ & $ (K_{3} \sqcup K_{3}^c )\vee K_{9}$ & &\\
&$\mathbf{0}\mathbf{1}\mathbf{0}^5 \mathbf{1}^{7}$ & $ (K_{2} \sqcup K_{5}^c )\vee K_{8}$ & Yes&\\
&$\mathbf{0}\mathbf{1}^2 \mathbf{0}^4 \mathbf{1}^8$ & $ (K_{3} \sqcup K_{4}^c )\vee K_{8}$ & Yes&\\
\hline
16&$\mathbf{0}\mathbf{1}^{15}$ & $ K_{16}$ & Yes&\\
&$\mathbf{0}^2 \mathbf{1}^{14}$ & $ K_{2}^c \vee K_{14}$ & Yes& Yes\\
&$\mathbf{0}^3 \mathbf{1}^{13}$ & $ K_{3}^c \vee K_{13}$ & Yes&\\
&$\mathbf{0}^4 \mathbf{1}^{12}$ & $ K_{4}^c \vee K_{12}$ & Yes&\\
&$\mathbf{0}^5 \mathbf{1}^{11}$ & $ K_{5}^c \vee K_{11}$ & &\\
&$\mathbf{0}^6 \mathbf{1}^{10}$ & $ K_{6}^c \vee K_{10}$ & &\\
&$\mathbf{0}^7 \mathbf{1}^9$ & $ K_{7}^c \vee K_{9}$ & Yes&\\
&$\mathbf{0}^8 \mathbf{1}^8$ & $ K_{8}^c \vee K_{8}$ & Yes&\\
&$\mathbf{0}\mathbf{1}\mathbf{0}^2 \mathbf{1}^{12}$ & $ (K_{2} \sqcup K_{2}^c) \vee K_{12}$ & Yes& Yes\\
&$\mathbf{0}\mathbf{1}\mathbf{0}^3 \mathbf{1}^{11}$ & $ (K_{2} \sqcup K_{3}^c )\vee K_{11}$ & &\\
&$\mathbf{0}\mathbf{1}\mathbf{0}^4 \mathbf{1}^{10}$ & $ (K_{2} \sqcup K_{4}^c )\vee K_{10}$ & &\\
&$\mathbf{0}\mathbf{1}^2 \mathbf{0}^3 \mathbf{1}^{10}$ & $( K_{3} \sqcup K_{3}^c) \vee K_{10}$ & &\\
&$\mathbf{0}\mathbf{1}\mathbf{0}^{5} \mathbf{1}^9$ & $ (K_{2} \sqcup K_{5}^c )\vee K_{9}$ & Yes&\\
&$\mathbf{0}\mathbf{1}^2 \mathbf{0}^4 \mathbf{1}^9$ & $ (K_{3} \sqcup K_{4}^c )\vee K_{9}$ & Yes&\\
&$\mathbf{0}\mathbf{1}\mathbf{0}^6\mathbf{1}^8$ & $ (K_{2} \sqcup K_{6}^c )\vee K_{8}$ & Yes& Yes\\
&$\mathbf{0}\mathbf{1}^2\mathbf{0}^5 \mathbf{1}^8$ & $ (K_{3} \sqcup K_{5}^c )\vee K_{8}$ & Yes&\\
&$\mathbf{0}\mathbf{1}^3 \mathbf{0}^4 \mathbf{1}^8$ & $ (K_{4} \sqcup K_{4}^c )\vee K_{8}$ & Yes&\\
&$\mathbf{0}^2 \mathbf{1}^2 \mathbf{0}^4 \mathbf{1}^8$ & $ ((K_{2}^c \vee K_{2}) \sqcup K_{4}^c) \vee K_{8}$ & Yes& Yes\\
\hline
17&$\mathbf{0}\mathbf{1}^{16}$ & $ K_{17}$ &Yes &\\
&$\mathbf{0}^2 \mathbf{1}^{15}$ & $ K_{2}^c \vee K_{15}$ &Yes &Yes\\
&$\mathbf{0}^{3} \mathbf{1}^{14}$ & $ K_{3}^c \vee K_{14}$ &Yes &\\
&$\mathbf{0}^{4} \mathbf{1}^{13}$ & $ K_{4}^c \vee K_{13}$ &Yes &\\
&$\mathbf{0}^5 \mathbf{1}^{12}$ & $ K_{5}^c \vee K_{12}$ & &\\
&$\mathbf{0}^6 \mathbf{1}^{11}$ & $ K_{6}^c \vee K_{11}$ & &\\
&$\mathbf{0}^{7} \mathbf{1}^{10}$ & $ K_{7}^c \vee K_{10}$ & &\\
&$\mathbf{0}^8 \mathbf{1}^{9}$ & $ K_{8}^c \vee K_{9}$ &Yes &\\
&$\mathbf{0}\mathbf{1}\mathbf{0}^2 \mathbf{1}^{13}$ & $ (K_{2} \sqcup K_{2}^c) \vee K_{13}$ &Yes &\\
&$\mathbf{0}\mathbf{1}\mathbf{0}^3 \mathbf{1}^{12}$ & $ (K_{2} \sqcup K_{3}^c )\vee K_{12}$ & &\\
&$\mathbf{0}\mathbf{1}\mathbf{0}^4 \mathbf{1}^{11}$ & $ (K_{2} \sqcup K_{4}^c )\vee K_{11}$ & &\\
&$\mathbf{0}\mathbf{1}^2 \mathbf{0}^3 \mathbf{1}^{11}$ & $( K_{3} \sqcup K_{3}^c) \vee K_{11}$ & &\\
&$\mathbf{0}\mathbf{1}\mathbf{0}^5 \mathbf{1}^{10}$ & $ (K_{2} \sqcup K_{5}^c )\vee K_{10}$ & &\\
&$\mathbf{0}\mathbf{1}^2 \mathbf{0}^4 \mathbf{1}^{10}$ & $( K_{3} \sqcup K_{4}^c) \vee K_{10}$ & &\\
&$\mathbf{0}\mathbf{1}\mathbf{0}^6 \mathbf{1}^9$ & $ (K_{2} \sqcup K_{6}^c )\vee K_{9}$ &Yes &\\
&$\mathbf{0}\mathbf{1}^2 \mathbf{0}^5 \mathbf{1}^9$ & $( K_{3} \sqcup K_{5}^c )\vee K_{9}$ &Yes &\\
&$\mathbf{0}\mathbf{1}^3 \mathbf{0}^4 \mathbf{1}^9$ & $( K_{4} \sqcup K_{4}^c )\vee K_{9}$ &Yes &\\
&$\mathbf{0}^2 \mathbf{1}^2 \mathbf{0}^4 \mathbf{1}^{9}$ & $(( K_{2}^c \vee K_{2}) \sqcup K_{4}^c )\vee K_{9}$ &Yes &\\
\hline 
18&$\mathbf{0}\mathbf{1}^{17}$ & $ K_{18}$ &Yes &\\
&$\mathbf{0}^2 \mathbf{1}^{16}$ & $ K_{2}^c \vee K_{16}$ &Yes &Yes\\
&$\mathbf{0}^3 \mathbf{1}^{15}$ & $ K_{3}^c \vee K_{15}$ &Yes &\\
&$\mathbf{0}^4 \mathbf{1}^{14}$ & $ K_{4}^c \vee K_{14}$ &Yes &\\
&$\mathbf{0}^5 \mathbf{1}^{13}$ & $ K_{5}^c \vee K_{13}$ & &\\
&$\mathbf{0}^6 \mathbf{1}^{12}$ & $ K_{6}^c \vee K_{12}$ & &\\
&$\mathbf{0}^7 \mathbf{1}^{11}$ & $ K_{7}^c \vee K_{11}$ & &\\
&$\mathbf{0}^8 \mathbf{1}^{10}$ & $ K_{8}^c \vee K_{10}$ &Yes &\\
&$\mathbf{0}^9 \mathbf{1}^9$ & $ K_{9}^c \vee K_{9}$ &Yes &\\
&$\mathbf{0}\mathbf{1}\mathbf{0}^2 \mathbf{1}^{14}$ & $ (K_{2} \sqcup K_{2}^c) \vee K_{14}$ &Yes &\\
&$\mathbf{0}\mathbf{1}\mathbf{0}^3 \mathbf{1}^{13}$ & $ (K_{2} \sqcup K_{3}^c )\vee K_{13}$ & &\\
&$\mathbf{0}\mathbf{1} \mathbf{0}^4 \mathbf{1}^{12}$ & $ (K_{2} \sqcup K_{4}^c )\vee K_{12}$ & &\\
&$\mathbf{0}\mathbf{1}^2 \mathbf{0}^3 \mathbf{1}^{12}$ & $( K_{3} \sqcup K_{3}^c) \vee K_{12}$ & &\\
&$\mathbf{0}\mathbf{1} \mathbf{0}^5 \mathbf{1}^{11}$ & $ (K_{2} \sqcup K_{5}^c )\vee K_{11}$ & &\\
&$\mathbf{0}\mathbf{1}^2 \mathbf{0}^4 \mathbf{1}^{11}$ & $( K_{3} \sqcup K_{4}^c) \vee K_{11}$ & &\\
&$\mathbf{0}\mathbf{1} \mathbf{0}^6 \mathbf{1}^{10}$ & $ (K_{2} \sqcup K_{6}^c )\vee K_{10}$ &Yes &\\
&$\mathbf{0}\mathbf{1}^2 \mathbf{0}^5 \mathbf{1}^{10}$ & $( K_{3} \sqcup K_{5}^c) \vee K_{10}$ &Yes &\\
&$\mathbf{0}\mathbf{1}^3 \mathbf{0}^4 \mathbf{1}^{10}$ & $( K_{4} \sqcup K_{4}^c )\vee K_{10}$ & Yes&\\
&$\mathbf{0}^2 \mathbf{1}^2 \mathbf{0}^4 \mathbf{1}^{10}$ & $( (K_{2}^c \vee K_{2} )\sqcup K_{4}^c) \vee K_{10}$ &Yes &\\
&$\mathbf{0}\mathbf{1}\mathbf{0}^7 \mathbf{1}^9$ & $ (K_{2} \sqcup K_{7}^c )\vee K_{9}$ & Yes&\\
&$\mathbf{0}\mathbf{1}^2 \mathbf{0}^6 \mathbf{1}^9$ & $( K_{3} \sqcup K_{6}^c )\vee K_{9}$ & &\\
&$\mathbf{0}\mathbf{1}^3 \mathbf{0}^5 \mathbf{1}^9$ & $ (K_{4} \sqcup K_{5}^c )\vee K_{9}$ & Yes &\\
&$\mathbf{0}^2 \mathbf{1}^2 \mathbf{0}^5 \mathbf{1}^9$ & $( (K_{2}^c \vee K_{2}) \sqcup K_{5}^c )\vee K_{9}$ &Yes &\\
\hline
19&$\mathbf{0}\mathbf{1}^{18}$ & $ K_{19}$ &Yes &\\
&$\mathbf{0}^2 \mathbf{1}^{17}$ & $ K_{2}^c \vee K_{17}$ &Yes & Yes\\
&$\mathbf{0}^3 \mathbf{1}^{16}$ & $ K_{3}^c \vee K_{16}$ &  &\\
&$\mathbf{0}^4 \mathbf{1}^{15}$ & $ K_{4}^c \vee K_{15}$ &Yes &\\
&$\mathbf{0}^6 \mathbf{1}^{13}$ & $ K_{6}^c \vee K_{13}$ & &\\
&$\mathbf{0}^7 \mathbf{1}^{12}$ & $ K_{7}^c \vee K_{12}$ & &\\
&$\mathbf{0}^8 \mathbf{1}^{11}$ & $ K_{8}^c \vee K_{11}$ & &\\
&$\mathbf{0}^9 \mathbf{1}^{10}$ & $ K_{9}^c \vee K_{10}$ &Yes &\\
&$\mathbf{0} \mathbf{1} \mathbf{0}^2 \mathbf{1}^{15}$ & $ (K_{2} \sqcup K_{2}^c) \vee K_{15}$ &Yes &\\
&$\mathbf{0} \mathbf{1} \mathbf{0}^3 \mathbf{1}^{14}$ & $ (K_{2} \sqcup K_{3}^c) \vee K_{14}$ & &\\
&$\mathbf{0} \mathbf{1} \mathbf{0}^4 \mathbf{1}^{13}$ & $ (K_{2} \sqcup K_{4}^c) \vee K_{13}$ & &\\
&$\mathbf{0} \mathbf{1}^{2} \mathbf{0}^3 \mathbf{1}^{13}$ & $ (K_{3} \sqcup K_{3}^c) \vee K_{13}$ & &\\
&$\mathbf{0} \mathbf{1} \mathbf{0}^5 \mathbf{1}^{12}$ & $ (K_{2} \sqcup K_{5}^c) \vee K_{12}$ & &\\
&$\mathbf{0} \mathbf{1}^{2} \mathbf{0}^4 \mathbf{1}^{12}$ & $ (K_{3} \sqcup K_{4}^c) \vee K_{12}$ & &\\
&$\mathbf{0} \mathbf{1} \mathbf{0}^6 \mathbf{1}^{11}$ & $ (K_{2} \sqcup K_{6}^c) \vee K_{11}$ & &\\
&$\mathbf{0} \mathbf{1}^{2} \mathbf{0}^5 \mathbf{1}^{11}$ & $ (K_{3} \sqcup K_{5}^c) \vee K_{11}$ & &\\
&$\mathbf{0} \mathbf{1}^{3} \mathbf{0}^4 \mathbf{1}^{11}$ & $ (K_{4} \sqcup K_{4}^c) \vee K_{11}$ & &\\
&$\mathbf{0}^{2} \mathbf{1}^{2} \mathbf{0}^4 \mathbf{1}^{11}$ & $((K_{2}^c \vee K_{2}) \sqcup K_{4}^c) \vee K_{11}$ & &\\
&$\mathbf{0} \mathbf{1} \mathbf{0}^7 \mathbf{1}^{10}$ & $(K_{2} \sqcup K_{7}^c) \vee K_{10}$ &Yes &\\
&$\mathbf{0} \mathbf{1}^{2} \mathbf{0}^6 \mathbf{1}^{10}$ & $(K_{3} \sqcup K_{6}^c) \vee K_{10}$ & &\\
&$\mathbf{0} \mathbf{1}^{3} \mathbf{0}^5 \mathbf{1}^{10}$ & $(K_{4} \sqcup K_{5}^c) \vee K_{10}$ &Yes &\\
&$\mathbf{0}^{2} \mathbf{1}^{2} \mathbf{0}^5 \mathbf{1}^{10}$ & $((K_{2}^{c} \vee K_{2}) \sqcup K_{5}^{c}) \vee K_{10}$ &Yes  &\\
\hline
20&$\mathbf{0}\mathbf{1}^{19}$ & $ K_{20}$ &Yes &\\
&$\mathbf{0}^2 \mathbf{1}^{18}$ & $ K_{2}^c \vee K_{18}$ &Yes & Yes\\
&$\mathbf{0}^3 \mathbf{1}^{17}$ & $ K_{3}^c \vee K_{17}$ &  &\\
&$\mathbf{0}^4 \mathbf{1}^{16}$ & $ K_{4}^c \vee K_{16}$ &Yes &\\
&$\mathbf{0}^5 \mathbf{1}^{15}$ & $ K_{5}^c \vee K_{15}$ &Yes &\\
&$\mathbf{0}^6 \mathbf{1}^{14}$ & $ K_{6}^c \vee K_{14}$ & &\\
&$\mathbf{0}^7 \mathbf{1}^{13}$ & $ K_{7}^c \vee K_{13}$ & &\\
&$\mathbf{0}^8 \mathbf{1}^{12}$ & $ K_{8}^c \vee K_{12}$ &  &\\
&$\mathbf{0}^9 \mathbf{1}^{11}$ & $ K_{9}^c \vee K_{11}$ & Yes &\\
&$\mathbf{0}^{10} \mathbf{1}^{10}$ & $ K_{10}^c \vee K_{10}$ & Yes &\\

&$\mathbf{0} \mathbf{1} \mathbf{0}^2 \mathbf{1}^{16}$ & $ (K_{2} \sqcup K_{2}^c) \vee K_{16}$ &Yes &Yes\\
&$\mathbf{0} \mathbf{1} \mathbf{0}^3 \mathbf{1}^{15}$ & $ (K_{2} \sqcup K_{3}^c) \vee K_{15}$ &Yes &\\
&$\mathbf{0} \mathbf{1} \mathbf{0}^4 \mathbf{1}^{14}$ & $ (K_{2} \sqcup K_{4}^c) \vee K_{14}$ & &\\
&$\mathbf{0} \mathbf{1}^{2} \mathbf{0}^3 \mathbf{1}^{14}$ & $ (K_{3} \sqcup K_{3}^c) \vee K_{14}$ & &\\
&$\mathbf{0} \mathbf{1} \mathbf{0}^5 \mathbf{1}^{13}$ & $ (K_{2} \sqcup K_{5}^c) \vee K_{13}$ & &\\
&$\mathbf{0} \mathbf{1}^{2} \mathbf{0}^4 \mathbf{1}^{13}$ & $ (K_{3} \sqcup K_{4}^c) \vee K_{13}$ & &\\
&$\mathbf{0} \mathbf{1} \mathbf{0}^6 \mathbf{1}^{12}$ & $ (K_{2} \sqcup K_{6}^c) \vee K_{12}$ & &\\
&$\mathbf{0} \mathbf{1}^{2} \mathbf{0}^5 \mathbf{1}^{12}$ & $ (K_{3} \sqcup K_{5}^c) \vee K_{12}$ & &\\
&$\mathbf{0} \mathbf{1}^{3} \mathbf{0}^4 \mathbf{1}^{12}$ & $ (K_{4} \sqcup K_{4}^c) \vee K_{12}$ & &\\
&$\mathbf{0}^{2} \mathbf{1}^{2} \mathbf{0}^4 \mathbf{1}^{12}$ & $((K_{2}^c \vee K_{2}) \sqcup K_{4}^c) \vee K_{12}$ & &Yes\\
&$\mathbf{0} \mathbf{1} \mathbf{0}^7 \mathbf{1}^{11}$ & $(K_{2} \sqcup K_{7}^c) \vee K_{11}$ &Yes &\\
&$\mathbf{0} \mathbf{1}^{2} \mathbf{0}^6 \mathbf{1}^{11}$ & $(K_{3} \sqcup K_{6}^c) \vee K_{11}$ & &\\
&$\mathbf{0} \mathbf{1}^{3} \mathbf{0}^5 \mathbf{1}^{11}$ & $(K_{4} \sqcup K_{5}^c) \vee K_{11}$ &Yes &\\
&$\mathbf{0}^{2} \mathbf{1}^{2} \mathbf{0}^5 \mathbf{1}^{11}$ & $((K_{2}^{c} \vee K_{2}) \sqcup K_{5}^{c}) \vee K_{11}$ &Yes  &\\
&$\mathbf{0} \mathbf{1} \mathbf{0}^8 \mathbf{1}^{10}$ & $(K_{2} \sqcup K_{8}^{c}) \vee K_{10}$ &Yes  &\\
&$\mathbf{0} \mathbf{1}^{2} \mathbf{0}^7 \mathbf{1}^{10}$ & $(K_{3} \sqcup K_{7}^{c}) \vee K_{10}$ &  &\\
&$\mathbf{0} \mathbf{1}^{3} \mathbf{0}^6 \mathbf{1}^{10}$ & $(K_{4} \sqcup K_{6}^{c}) \vee K_{10}$ & Yes &\\
&$\mathbf{0}^{2} \mathbf{1}^{2} \mathbf{0}^6 \mathbf{1}^{10}$ & $((K_{2}^{c} \vee K_{2}) \sqcup  K_{6}^{c}) \vee K_{10}$ & Yes &\\
&$\mathbf{0} \mathbf{1}^{4} \mathbf{0}^5 \mathbf{1}^{10}$ & $ (K_{5} \sqcup K_{5}^c) \vee K_{10}$ &Yes  &\\
&$\mathbf{0}^{2} \mathbf{1}^{3} \mathbf{0}^5 \mathbf{1}^{10}$ & $((K_{2}^{c} \vee K_{3}) \sqcup  K_{5}^{c}) \vee K_{10}$ & Yes &\\
\hline

\end{longtable}
\end{center}

\bibliographystyle{plain}
\bibliography{mybibfile}

@article{WHD,
  title={Weakly Hadamard diagonalizable graphs},
  author={Adm, M. and Almuhtaseb, K. and Fallat, S. and Meagher, K. and Nasserasr, S. and Shirazi, M.N. and Razafimahatratra, A.S.},
  journal={Linear Algebra Appl.},
  volume={610},
  pages={86--119},
  year={2021},
  publisher={Elsevier}
}

@article{antiregular,
  title={The role of the anti-regular graph in the spectral analysis of threshold graphs},
  author={Aguilar, C.O. and Ficarra, M. and Schurman, N. and Sullivan, B.},
  journal={Linear Algebra Appl.},
  volume={588},
  pages={210--223},
  year={2020},
  publisher={Elsevier}
}

@article{AL21,
  title={On graphs with adjacency and signless Laplacian matrices eigenvectors entries in $\{$- 1,+ 1$\}$},
  author={Alencar, J. and de Lima, L.},
  journal={Linear Algebra Appl.},
  volume={614},
  pages={301--315},
  year={2021},
  publisher={Elsevier}
}

@article{ALN23,
  title={On graphs with eigenvectors in $\{$- 1, 0, 1$\}$ and the max k-cut problem},
  author={Alencar, J. and de Lima, L. and Nikiforov, V.},
  journal={Linear Algebra Appl.},
  volume={663},
  pages={222--240},
  year={2023},
  publisher={Elsevier}
}

@article{AAGK06,
  title={$\{$- 1, 0, 1$\}$-basis for the null space of a forest},
  author={Akbari, S. and Alipour, A. and Ghorbani, Ebrahim and Khosrovshahi, G.B.},
  journal={Linear Algebra Appl.},
  volume={414},
  number={2-3},
  pages={506--511},
  year={2006},
  publisher={Elsevier}
}

@article{barik2011,
  title={On Hadamard diagonalizable graphs},
  author={Barik, S. and Fallat, S. and Kirkland, S.},
  journal={Linear Algebra Appl.},
  volume={435},
  number={8},
  pages={1885--1902},
  year={2011},
  publisher={Elsevier}
}

@article{JP25,
  title={Laplacian $\{$- 1, 0, 1$\}$-and $\{$- 1, 1$\}$-diagonalizable graphs},
  author={Johnston, N. and Plosker, S.},
  journal={Linear Algebra Appl.},
  volume={704},
  pages={309--339},
  year={2025},
  publisher={Elsevier}
}

@article{Caputo,
  title={On graph Laplacian eigenvectors with components in $\{$- 1, 0, 1$\}$},
  author={Caputo, J.G. and Khames, I. and Knippel, A.},
  journal={Discrete Appl. Math.},
  volume={269},
  pages={120--129},
  year={2019},
  publisher={Elsevier}
}

@article{MDTL24,
  title={A Laplacian eigenbasis for threshold graphs},
  author={Macharete, R. and Del-Vecchio, R. and Teixeira, H. and de Lima, L.},
  journal={Spec. Matrices},
  volume={12},
  number={1},
  pages={20240029},
  year={2024},
  publisher={De Gruyter}
}

@book{MP95,
  title={Threshold graphs and related topics},
  author={Mahadev, N.V.R. and Peled, U.N.},
  volume={56},
  year={1995},
  publisher={Elsevier}
}

@article{WeakMatrices,
  title={Weakly Hadamard diagonalizable graphs and quantum state transfer},
  author={McLaren, D. and Monterde, H. and Plosker, S.},
  journal={Linear Multilinear Algebra},
  volume={73},
  number={17},
  pages={3763--3790},
  year={2025},
  publisher={Taylor \& Francis}
}

@article{Merris94,
  title={Degree maximal graphs are Laplacian integral},
  author={Merris, R.},
  journal={Linear Algebra Appl.},
  volume={199},
  pages={381--389},
  year={1994},
  publisher={Elsevier}
}

@article{SS07,
  title={On simply structured kernel bases of unicyclic graphs},
  author={Sander, T. and Sander, J.W.},
  journal={AKCE Int. J. Graphs Comb.},
  volume={4},
  number={1},
  pages={61--82},
  year={2007},
  publisher={Taylor \& Francis}
}

@article{SS08,
  title={On certain eigenspaces of cographs},
  author={Sander, T.},
  journal={Electron. J. Comb.},
  pages={R140--R140},
  year={2008}
}

@article{SS09,
  title={Tree decomposition by eigenvectors},
  author={Sander, T. and Sander, J.W.},
  journal={Linear Algebra Appl.},
  volume={430},
  number={1},
  pages={133--144},
  year={2009},
  publisher={Elsevier}
}

@article{kirkland2011spin,
  title={Spin-system dynamics and fault detection in threshold networks},
  author={Kirkland, S. and Severini, S.},
  journal={Phys. Rev. A},
  volume={83},
  number={1},
  pages={012310},
  year={2011},
  publisher={APS}
}

@article{kirkland2020fractional,
  title={Fractional revival of threshold graphs under Laplacian dynamics},
  author={Kirkland, S. and Zhang, X.},
  year={2020},
  journal={Discuss. Math. Graph Theory}
}

@article{monterde2026new,
  title={New results in vertex sedentariness},
  author={Monterde, H.},
  journal={Discrete Math.},
  volume={349},
  number={4},
  pages={114959},
  year={2026},
  publisher={Elsevier}
}

@article{godsil2025perfect,
  title={Perfect state transfer between real pure states},
  author={Godsil, C. and Kirkland, S. and Monterde, H.},
  journal={SIAM J. Matrix Anal. Appl.},
  volume={46},
  number={3},
  pages={2093--2115},
  year={2025},
  publisher={SIAM}
}

@article{godsil2012state,
  title={State transfer on graphs},
  author={Godsil, C.},
  journal={Discrete Math.},
  volume={312},
  number={1},
  pages={129--147},
  year={2012},
  publisher={Elsevier}
}

@article{kim2024generalization,
  title={A generalization of quantum pair state transfer},
  author={Kim, S. and Monterde, H. and Ahmadi, B.and Chan, A. and Kirkland, S. and Plosker, S.},
  journal={Quantum Inf. Process.},
  year={2024},
  volume={23},
  pages = {369}
}

@article{chen2020pair,
  title={Pair state transfer},
  author={Chen, Q. and Godsil, C.},
  journal={Quantum Inf. Process.},
  volume={19},
  number={9},
  pages={321},
  year={2020},
  publisher={Springer}
}

@article{godsil2025quantum,
  title={Quantum walks on finite and bounded infinite graphs},
  author={Godsil, C. and Kirkland, S. and Mohapatra, S. and Monterde, H. and Pal, H.},
  journal={arXiv preprint arXiv:2510.05306},
  year={2025}
}

@article{breen,
  title={Hadamard Diagonalizable Graphs of Order at Most 36},
  author={Breen, J. and Butler, S. and Fuentes, M. and Lidick{\`y}, B. and Phillips, Michael and Riasanovksy, A. and Song, S. and Villagr{\'a}n, R. and Wiseman, C. and Zhang, X.},
  journal={Electronic Journal of Combinatorics},
  volume={29},
  number={2},
  pages={P2--16},
  year={2022},
  publisher={Electron. J. Comb.}
}

@article{kirkland2026quantum,
  title={Quantum walks on join graphs},
  author={Kirkland, S. and Monterde, H.},
  journal={Discrete Math.},
  volume={349},
  number={3},
  pages={114832},
  year={2026},
  publisher={Elsevier}
}

@article{MONTERDE2025,
title = {Laplacian quantum walks on blow-up graphs},
journal = {Linear Algebra Appl.},
year = {2025},
issn = {0024-3795},
doi = {https://doi.org/10.1016/j.laa.2025.10.001},
url = {https://www.sciencedirect.com/science/article/pii/S0024379525004082},
author = {H. Monterde and H. Pal and S. Kirkland}
}

@article{Monterde2022,
archivePrefix = {arXiv},
arxivId = {2111.01265},
author = {Monterde, H.},
doi = {10.13001/ela.2022.6721},
issn = {10813810},
journal = {Electron. J. Linear Algebra},
pages = {494--518},
title = {{Strong cospectrality and twin vertices in weighted graphs}},
volume = {38},
year = {2022}
}

@article{alvir2016perfect,
  title={Perfect state transfer in Laplacian quantum walk},
  author={Alvir, R. and Dever, S. and Lovitz, B. and Myer, J. and Tamon, C. and Xu, Y. and Zhan, H.},
  journal={J. Algebr. Comb.},
  volume={43},
  number={4},
  pages={801--826},
  year={2016},
  publisher={Springer}
}

\end{document}